\documentclass[12pt,draft]{amsart}
\usepackage{amsmath,amsthm,latexsym,amscd,amsbsy,amssymb,fleqn,leqno}
\chardef\bslash=`\\ % p. 424, TeXbook

\makeatletter
\def\verbatim{\interlinepenalty\@M \@verbatim
  \leftskip\@totalleftmargin\advance\leftskip2pc
  \frenchspacing\@vobeyspaces \@xverbatim}
\makeatother
\makeatletter
  \def\dgt@k{\dg@DX=-3 \dg@DY=2 \dg@SIZE=3}
\makeatother

\makeatletter
  \def\dgt@kk{\dg@DX=3 \dg@DY=-1 \dg@SIZE=3}%
\makeatother

\theoremstyle{plain}
\newtheorem{thm}{Theorem}[section]
\newtheorem{cor}[thm]{Corollary}
\newtheorem{lem}[thm]{Lemma}
\newtheorem{pro}[thm]{Proposition}

\theoremstyle{definition}

\numberwithin{equation}{section}

\newcounter{rmnum}

\def\symbolnote#1#2{\let\thefootn=\thefootnote%
\renewcommand{\thefootnote}{\fnsymbol{footnote}}%
\footnotemark[#1]%
\footnotetext[#1]{#2}%
\let\thefootnote=\thefootn
}

\newfont{\bbb}{msbm10 scaled \magstep1}
\newfont{\bbc}{msbm8 scaled \magstep0}

\newcommand{\R}{\mbox{\bbb R}}

\begin{document}

%%%%%%% Begin Topmatter %%%%%%%%%%

\title{Linear operators with compact supports, probability measures and Milyutin maps}

\author{Vesko Valov}
\address{Department of Computer Science and Mathematics, Nipissing University,
100 College Drive, P.O. Box 5002, North Bay, ON, P1B 8L7, Canada}
\email{veskov@nipissingu.ca}
\thanks{The author was partially supported by his NSERC grant 261914-08.}



%%%%%%% End topmatter %%%%%%%%%

\begin{abstract}
The notion of a regular operator with compact supports between
function spaces is introduced. On that base we obtain a
characterization of absolute extensors for $0$-dimensional spaces in
terms of regular extension operators having compact supports.
Milyutin maps are also considered and it is established that some
topological properties, like paracompactness, metrizability and
$k$-metrizability, are preserved under Milyutin maps.
\end{abstract}

\maketitle

\markboth{V.~Valov}{Probability measures}

%%%%%%%%%%%%%%%%%%%%%%%%%%%%%%%%%%%%%%%%%%%%%%%%%%%%%%%%%%%%

\section{Introduction}

In this paper we assume that all topological spaces are Tychonoff.
The main concept is that one of a linear map between function spaces
with compact supports. Let $u\colon C(X,E)\to C(Y,E)$ be a linear
map, where $C(X,E)$ is the set of all continuous functions from $X$
into a locally convex linear space $E$. We say that $u$ {\em has
compact supports} if for every $y\in Y$ the linear map $T(y)\colon
C(X,E)\to E$, defined by $T(y)(h)=u(h)(y)$, $h\in C(X,E)$, has a
compact support in $X$. Here, the support of a linear map $\mu\colon
C(X,E)\to E$ is the set $s(\mu)$ of all $x\in\beta X$ such that for
every neighborhood $U$ of $x$ in $\beta X$ there exists $h\in
C(X,E)$ with $(\beta h)(\beta X-U)=0$ and $\mu(h)\neq 0$. Recall
that $\beta X$ is the \v{C}ech-Stone compactification of $X$ and
$\beta h\colon\beta X\to \beta E$ the extension of $h$. Obviously,
$s(\mu)\subset \beta X$ is closed, so compact. When $s(\mu)\subset
X$, $\mu$ is said to have a compact support. In a similar way we
define a linear map with compact supports when consider the bounded
function sets $C^*(X,E)$ and $C^*(Y,E)$ (if $E$ is the real line
$\mathbb R$, we simply write $C(X)$ and $C^*(X)$). If all $T(y)$ are
{\em regular linear maps}, i.e., $T(y)(h)$ is contained in the
closed convex hull $\overline{~conv h(X)}$ of $h(X)$ in $E$, then
$u$ is called a {\em regular operator}.

Haydon \cite{ha} proved that Dugundji spaces introduced by
Pelczynski \cite{p} coincides with the absolute extensors for
$0$-dimensional compact spaces (br., $X\in AE(0)$). Later Chigogidze
\cite{ch2} provided a more general definition of $AE(0)$-spaces  in
the class of all Tychonoff spaces. The notion of linear operators
with compact supports arose from the attempt to find a
characterization of $AE(0)$-spaces similar to the Pelczynski
definition of Dugundji spaces. Here is this characterization (see
Theorems 4.1-4.2). {\em For any space $X$ the following conditions
are equivalent: $(i)$ $X$ is an $AE(0)$-space; $(ii)$ for every
$C$-embedding of $X$ in a space $Y$ there exists a regular extension
operator $u\colon C(X)\to C(Y)$ with compact supports; $(iii)$ for
every $C$-embedding of $X$ in a space $Y$ there exists a regular
extension operator $u\colon C^*(X)\to C^*(Y)$ with compact supports;
$(iv)$ for any $C$-embedding of $X$ in a space $Y$ and any complete
locally convex space $E$ there exists a regular extension operator
$u\colon C^*(X,E)\to C^*(Y,E)$ with compact supports}.

It is easily seen that $u\colon C(X,E)\to C(Y,E)$ (resp., $u\colon
C^*(X,E)\to C^*(Y,E)$) is a regular extension operator with compact
supports iff there exists a continuous map $T\colon Y\to P_c(X,E)$
(resp., $T\colon Y\to P_c^*(X,E)$) such that $T(y)$ is the Dirac
measure $\delta_y$ at $y$ for all $y\in X$. Here, $P_c(X,E)$ (resp.,
$P_c^*(X,E)$) is the space of all regular linear maps $\mu:C(X,E)\to
E$ (resp., $\mu:C^*(X,E)\to E$) with compact supports equipped with
the pointwise convergence topology (we write $P_c(X)$ and $P_c^*(X)$
when $E=\mathbb R$). Section 2 is devoted to properties of the
functors $P_c$ and $P_c^*$ (actually, $P_c^*$ is the well known
functor $P_\beta$ \cite{ch1} of all probability measures on $\beta
X$ whose supports are contained in $X$). It appears that {\em
$P_c(X)$ is homeomorphic to the closed convex hull of $e_X(X)$ in
$\mathbb R^{C(X)}$ provided $X$ is realcompact, where $e_X$ is the
standard embedding of $X$ into $\mathbb R^{C(X)}$} (Proposition
2.4), and {\em $P_c(X)$ is metrizable iff $X$ is a metric compactum}
(Proposition 2.5(ii)).

In Section 3 we consider regular averaging operators with compact
support and Milyutin maps. Milyutin maps between compact spaces were
introduced by Pelczynski \cite{p}. There are different definitions
of Milyutin maps in the non-compact case, see \cite{at}, \cite{rss}
and \cite{vv4}. We say that a surjection $f\colon X\to Y$ is {\em a
Milyutin map} if $f$ admits a regular averaging operator $u:C(X)\to
C(Y)$ having compact supports. This is equivalent to the existence
of a map $T:Y\to P_c(X)$ such that $f^{-1}(y)$ contains the support
of $T(y)$ for all $y\in Y$. It is shown, for example, that {\em for
every product $Y$ of metric spaces there is a $0$-dimensional
product $X$ of metric spaces and a perfect Milyutin map $f\colon
X\to Y$} (Corollary 3.10). Moreover, {\em every $p$-paracompact
space is an image under a perfect Milyutin map of a $0$-dimensional
$p$-paracompact space} (Corollary 3.11).

In the last Section 5 we prove that some topological properties are
preserved under Milyutin maps. These properties include
paracompactness, collectionwise normality, (complete) metrizability,
stratifiability, $\delta$-metrizability and $k$-metrizability. In
particular, we provide a positive answer to a question of Shchepin
\cite{sc3} whether {\em every $AE(0)$-space is $k$-metrizable} (see
Corollary 5.5).

Some of the result presented here were announced in \cite{vv3}
without proofs.

%%%%%%%%%%%%%%%%%%%%%%%%%%%%%%
%%%%%%%%%%%%%%%%%%%%%%%%%%%%%%

\section{Measure spaces}
Everywhere in this section $E,F$ stand for locally convex linear
topological spaces and $C(X,E)$ is the set of all continuous maps
from a space $X$ into $E$. By $C^*(X,E)$ we denote the bounded
elements of $C(X,E)$. Let $\mu\colon C(X,E)\to F$ (resp., $\mu\colon
C^*(X,E)\to F$) be a linear map. The support of $\mu$ is defined as
the set $s(\mu)$ (resp., $s^*(\mu)$) of all $x\in\beta X$ such that
for every neighborhood $U$ of $x$ in $\beta X$ there exists $f\in
C(X,E)$ (resp., $f\in C^*(X,E)$) with $(\beta f)(\beta X-U)=0$ and
$\mu(f)\neq 0$, see \cite {vv1}. Obviously, $s(\mu)$ and $s^*(\mu)$
are closed in $\beta X$, so compact. Let us note that in the above
definition $(\beta f)(\beta X-U)=0$ is equivalent to $f(X-U)=0$. We
also use $s^*(\mu)$ to denote the support of the restriction
$\mu|C^*(C,E)$ when $\mu$ is defined on $C(X,E)$ (in this case we
have $s^*(\mu)\subset s(\mu)$).

\begin{lem}
Let $\mu$ be a linear map from $C(X,E)$  $($resp., from
$C^*(X,E)$$)$ into F, where $E$ and $F$ are norm spaces.

\begin{itemize}
\item[(i)]  If $V$ a neighborhood of $s(\mu)$ $($resp., $s^*(\mu)$$)$, then
$\mu(f)=0$ for every $f\in C(X,E)$ $($resp., $f\in C^*(X,E)$$)$
with $(\beta f)(V)=0$.

\item[(ii)]
If the restriction $\mu|C^{*}(X,E)$ is continuous when $C^{*}(X,E)$
is equipped with the uniform topology, then $\mu(f)=0$ provided
$f\in C(X,E)$ $($resp., $f\in C^*(X,E)$$)$ and $(\beta f)(s(\mu))=0$
$($resp., $(\beta f)(s^*(\mu))=0$$)$.

\item[(iii)] In each of the following two cases $s(\mu)$ coincides with $s^*(\mu)$:
either $s(\mu)\subset X$ or $\mu$ is a non-negative linear
functional on $C(X)$.
\end{itemize}
\end{lem}

\begin{proof}
When $\mu$ is a linear map on $C(X,E)$, items (i) and (ii) were
established in \cite[Lemma 2.1]{vv1}; the case when $\mu$ is a
linear map on $C^{*}(X,E)$ can be  done by similar arguments.  To
prove (iii), we first suppose that  $s(\mu)\subset X$.  Then
$s^*(\mu)$ is the support of the restriction $\mu|C^{*}(X,E)$ and
$s^*(\mu)\subset s(\mu)$. So, we need to show that $s(\mu)\subset
s^*(\mu)$. For a given point $x\in s(\mu)$ and its neighborhood $U$
in $\beta X$ there exists $g\in C(X,E)$ with $g(X-U)=0$ and
$\mu(g)\neq 0$.  Because $g(s(\mu))\subset E$ is compact, we can
find $\epsilon>0$ such that $s(\mu)$ is contained in the set
$W=\{y\in X: ||g(y)||<\epsilon\}$,  where $||.||$ denotes the norm
in $E$.  Let $B_{\epsilon}=\{z\in E: ||z||\leq\epsilon\}$ and
$r\colon E\to B_{\epsilon}$ be a retraction (i.e., a continuous map
with $r(z)=z$ for every $z\in B_{\epsilon}$).  Then $h(y)=g(y)$ for
every $y\in W$, where  $h=r\circ g$. Hence, choosing an open set $V$
in $\beta X$ such that $V\cap X=W$,   we have $(\beta(h-g))(V)=0$.
Since $V$ is a neighborhood of $s(\mu)$, by (i), $\mu(h)=\mu(g)\neq
0$. Therefore, we found a map $h\in C^{*}(X,E)$ such that $\beta
h(\beta X-U)=0$ and $\mu(h)\neq 0$. This means that $x\in s^*(\mu)$.
So, $s(\mu)=s^*(\mu)$.

\medskip
Now, let $E=F=\R$ and $\mu$ be a non-negative linear functional on
$C(X)$.  Suppose there exists $x\in s(\mu)$ but $x\not\in s^*(\mu)$.
Then, for some neighborhood $U$ of $x$ in $\beta X$, we have

\medskip\noindent
(1)\hspace{0.3cm}$\mu(h)=0$ for every $h\in C^{*}(X)$ with $h(X-U)=0$.

\medskip\noindent
Since $x\in s(\mu)$, there exists $f\in C(X)$ such that $f(X-U)=0$
and $\mu(f)\neq 0$. Now, we use an idea from \cite[proof of Theorem
1]{he}.  We represent $f$ as the sum $f^{+}+f^{-}$, where
$f^{+}=\max\{f,0\}$ and $f^{-}=\min\{f,0\}$. Since both $f^{+}$ and
$f^{-}$ are 0 outside $U$ and $\mu(f)=\mu(f^{+})+\mu(f^{-})\neq 0$
implies that at least one of the numbers $\mu(f^{+})$ and
$\mu(f^{-})$ is not 0, we can assume that $f\geq 0$. By (1), $f$ is
not bounded. Therefore, there is a sequence
$\displaystyle\{x_n\}\subset X$ such that
$\displaystyle\{t_n=f(x_n):n\geq 1\}$ is an increasing and unbounded
sequence. We set $\displaystyle t_0=0$ and for every $n\geq 1$
define the function $\displaystyle f_n\in C^{*}(X)$ as follows:
$\displaystyle f_n(x)=0$ if $\displaystyle f(x)\leq t_{n-1}$,
$\displaystyle f_n(x)=f(x)-t_{n-1}$ if $\displaystyle
t_{n-1}<f(x)\leq t_n$ and $\displaystyle f_n(x)=t_n-t_{n-1}$
provided $\displaystyle f(x)>t_n$. Let also $\displaystyle
h_n=t_n\cdot f_n$ and $\displaystyle h=\sum_{n=1}^{\infty}h_n$. Then
$h$ is continuous and for every $n\geq 1$ we have

\medskip\noindent
(2)\hspace{0.3cm}$\displaystyle t_n\big(f-f_1-f_2-...-f_n\big)\leq h-h_1-h_2-...-h_n$.
\medskip\noindent

Since all $\displaystyle f_n$ and $\displaystyle h_n$ are bounded
and continuous functions satisfying   $\displaystyle
f_n(X-U)=h_n(X-U)=0$,   it follows from (1) that
$\displaystyle\mu(h_n)=\mu(f_n)=0$, $n\geq 1$. So, by (2),
$\displaystyle t_n\cdot\mu(f)\leq\mu(h)$ for every $n$. Hence,
$\mu(f)=0$ which is a contradiction.  Therefore, $s(\mu)=s^*(\mu)$.
\end{proof}

We say that a linear map $\mu$ on $C(X,E)$  (resp., on $C^{*}(X,E)$)
has a {\it compact support} if $s(\mu)\subset X$ (resp.,
$s^*(\mu)\subset X$). If $\mu$ takes values in $E$, then it is
called {\it regular} provided $\mu(f)$ belongs to the closure of the
convex hull $conv~f(X)$ of $f(X)$ for every $f\in C(X,E)$ (resp.,
$f\in C^{*}(X,E)$).
%We denote by $M^+(X,E)$ the set of all regular linear maps from $C(X,E)$ into $E$.
Below, $C_k(X,E)$ (resp., $C_k^{*}(X,E)$)  stands for the space
$C(X,E)$ (resp. $C^{*}(X,E)$) with the compact-open topology.

\begin{pro}
Let $E$ be a norm space.  A regular linear map $\mu$ on $C(X,E)$
$($resp., $C^{*}(X,E)$$)$ has a compact support in $X$ if and only
if $\mu$ is continuous on $C_k(X,E)$ $($resp., $C^*_k(X,E)$$)$.
\end{pro}

\begin{proof}
We consider only the case when $\mu$ is a map on $C(X,E)$, the other
one is similar. Suppose $s(\mu)=H\subset X$. Since $\mu$ is regular,
$\mu(f)\in\overline{conv~f(X)}$ for every $f\in C(X,E)$. This yields
$||\mu(f)||\leq ||f||$, $f\in C^{*}(X,E)$. Hence, the restriction
$\mu|C^{*}(X,E)$ is continuous with respect to the uniform topology.
So, by Lemma 2.1(ii), for every $f\in C(X,E)$ the value $\mu(f)$
depends only on the restriction $f|H$. Therefore, the linear map
$\nu\colon C(H,E)\to E$, $\nu(g)=\mu(\widetilde{g})$, where
$\widetilde{g}\in C(X,E)$ is any continuous extension of $g$, is
well defined. Note that such an extension $\widetilde{g}$ always
exists because $H\subset X$ is compact. Moreover, the restriction
map $\pi_H\colon C_k(X,E)\to C_k(H,E)$ is surjective and continuous.
Since $\mu=\nu\circ\pi_H$, $\mu$ would be continuous provided
$\nu\colon C_k(H,E)\to E$ is  so. Next claim implies that for every
$g\in C(H,E)$ we have $\nu(g)\in\overline{conv~g(H)}$ and
$||\nu(g)||\leq ||g||$, which guarantee the continuity of $\nu$.

\medskip\noindent
{\em Claim $1$.  $\mu(f)\in\overline{conv~f(H)}$ for every $f\in C(X,E)$}\\

Indeed, if $\mu(f)\not\in\overline{conv~f(H)}$ for some $f\in
C(X,E)$, then we can find a closed convex neighborhood $W$ of
$\overline{conv~f(H)}$ in $E$  and a function $h\in C(X,E)$ such
that $\mu(f)\not\in W$, $h(X)\subset W$ and $h(x)=f(x)$ for all
$x\in H$. As it was shown above, the last equality implies
$\mu(f)=\mu(h)$. Hence, $\mu(f)=\mu(h)\in\overline{conv~h(X)}\subset
W$, which is a contradiction.

Now, suppose $\mu\colon C_k(X,E)\to E$ is continuous. Then there
exists a compact set $K\subset X$ and $\epsilon>0$ such that
$||\mu(f)||<1$ for every $f\in C(X,E)$ with $\sup\{||f(x)||:x\in
K\}<\epsilon$. We claim that $s(\mu)\subset K$. Indeed, otherwise
there would be $x\in s(\mu)-K$, a neighborhood $U$ of $x$ in $\beta
X$ with $U\cap K=\varnothing$, and a function $g\in C(X,E)$ such
that $g(X-U)=0$ and $\mu(g)\neq 0$. Choose an integer $k$ with
$||\mu(kg)||\geq 1$. On the other hand, $kg(x)=0$ for every $x\in
K$. Hence, $||\mu(kg)||<1$, a contradiction.
\end{proof}

Now, for every space $X$ and a locally convex space $E$ let
$P_c(X,E)$ (resp., $P_c^*(X,E)$) denote the set of all regular
linear maps $\mu\colon C(X,E)\to E$ (resp., $\mu\colon C^*(X,E)\to
E$) with compact supports equipped with the weak (i.e. pointwise)
topology with respect to $C(X,E)$ (resp., $C^*(X,E)$). If $E$ is the
real line, we write $P_c(X)$ (resp., $P_c^*(X)$) instead of
$P_c(X,\mathbb R)$ (resp., $P_c^*(X,\mathbb R)$). It is easily seen
that a linear map $\mu\colon C(X)\to\mathbb R$ (resp., $\mu\colon
C^*(X)\to\mathbb R$) is regular if and only if $\mu$ is non-negative
and $\mu(1)=1$. If $h\colon X\to Y$ is a continuous map, then there
exists a map $P_c(h)\colon P_c(X)\to P_c(Y)$ defined by
$P_c(h)(\mu)(f)=\mu(f\circ h)$, where $\mu\in P_c(X)$ and $f\in
C(Y)$. Considering functions $f\in C^*(Y)$ in the above formula, we
can define a map $P_c^*(h)\colon P_c^*(X)\to P_c^*(Y)$. It is easily
seen that $s(P_c(h)(\mu))\subset h(s(\mu))$ (resp.,
$s^*(P_c^*(h)(\mu))\subset h(s^*(\mu))$) for every $\mu\in P_c(X)$
(resp., $\mu\in P_c^*(X)$). Moreover, $P_c(h_2\circ
h_1)=P_c(h_2)\circ P_c(h_1)$ and $P_c^*(h_2\circ
h_1)=P_c^*(h_2)\circ P_c^*(h_1)$ for any two maps $h_1\colon X\to Y$
and $h_2\colon Y\to Z$. Therefore, both $P_c$ and $P_c^*$ are
covariant functors in the category of all Tychonoff spaces and
continuous maps. Let us also note that if $X$ is compact then
$P_c(X)$ and $P_c^*(X)$ coincide with the space $P(X)$ of all
probability measures on $X$.

For every $x\in X$ we consider the Dirac's measure $\delta_x\in
P_c(X,E)$ defined by $\delta_x(f)=f(x)$, $f\in C(X,E)$. In a similar
way we define $\delta^*_x\in P_c^*(X,E)$. We also consider the maps
$i_X\colon X\to P_c(X,E)$, $i_X(x)=\delta_x$, and $i_X^*\colon X\to
P_c^*(X,E)$, $i_X(x)=\delta_x^*$. Next proposition is an easy
exercise.

\begin{pro}
\begin{itemize} Let $h\colon X\to
Y$be a map.
\item[(i)] The map $i_X:X\to P_c(X)$ is a closed
$C$-embedding, and $i_X^*:X\to P_c^*(X)$ is a closed
$C^*$-embedding;
\item[(ii)] The map $P_c(h)$   is a $($closed$)$ $C$-embedding
provided $h$ is a $($closed$)$ $C$-embedding;
\item[(iii)] The map $P_c^*(h)$ is a $($closed$)$ $C^*$-embedding
provided $h$ is a $($closed$)$ $C^*$-embedding.
\end{itemize}
\end{pro}

There exists a natural embedding $e_X\colon X\to\mathbb R^{C(X)}$,
$e_X(x)=(f(x))_{f\in C(X)}$. Denote by $M^+(X)$ the set of all
regular linear functionals on $C(X)$ with the pointwise topology and
consider the map $m_X\colon M^+(C)\to\mathbb R^{C(X)}$,
$m_X(\mu)=(\mu(f))_{f\in C(X)}$. It easily seen that $m_X$ is also
an embedding extending and $m_X(M^+(X))$ is a closed convex subset
of $\mathbb R^{C(X)}$. Moreover, $P_c(X)\subset M^+(X)$. It is well
known that for compact $X$ the space $P(X)$ is homeomorphic with the
convex closed hull of $e_X(X)$ in $\mathbb R^{C(X)}$. A similar fact
is true for $P_c(X)$.

\begin{pro}
If $X$ is realcompact, then $P_c(X)$ is homeomorphic to the closed
convex hull of $e_X(X)$ in $\mathbb R^{C(X)}$.
\end{pro}

\begin{proof}
Obviously, $m_X(P_c(X))$ is a convex subset of $\mathbb R^{C(X)}$
containing the set $conv~e_X(X)$. It suffices to show that
$m_X(P_c(X))$ coincides with the set $B=\overline{conv~e_X(X)}$.
Suppose $\mu\in P_c(X)$. By Lemma 2.1(ii) and Proposition 2.2, for
every $f\in C(X)$ the value $\mu(f)$ is determined by the
restriction $f|s(\mu)$. So, there exists an element $\nu\in
P(s(\mu))$ such that $\mu(f)=\nu(f|s(\mu))$, $f\in C(X)$ (see the
proof of Proposition 2.2). Since the set $P_f(s(\mu))$ of all
measures from $P(s(\mu))$ having finite supports is dense in
$P(s(\mu))$ \cite{ff}, there is a net $\{\nu_\alpha\}_{\alpha\in
A}\subset P_f(s(\mu))$ converging to $\nu$ in $P(s(\mu))$. Each
$\nu_\alpha$ can be identified with the measure $\mu_\alpha\in
P_c(X)$ defined by $\mu_\alpha(f)=\nu_\alpha(f|s(\mu))$, $f\in
C(X)$. Moreover, the net $\{\mu_\alpha\}_{\alpha\in A}$ converges to
$\mu$ in $P_c(X)$. Then $\{m_X(\mu_\alpha)\}_{\alpha\in A}\subset
conv~e_X(X)$ and converges to $m_X(\mu)$ in $\mathbb R^{C(X)}$. So,
$m_X(\mu)\in B$. In this way we obtained $m_X(P_c(X))\subset B$.

On the other hand, since $m_X(M^+(X))$ is a closed and convex subset
of $\mathbb R^{C(X)}$ containing $e_X(X)$, $B\subset m_X(M^+(X))$.
So, the elements of $B$ are of the form $m_X(\mu)$ with $\mu$ being
a regular linear functional on $C(X)$. Since $X$ is realcompact,
according to \cite[Theorem 18]{he}, any such a functional has a
compact support in $X$. Therefore, $B\subset m_X(P_c(X))$.
\end{proof}

There exists a natural continuous map $j_X\colon P_c(X)\to P_c^*(X)$
assigning to each $\mu\in P_c(X)$ the measure $\nu=\mu|C^*(X)$. By
Lemma 2.1 and Proposition 2.2, $s(\mu)=s^*(\nu)$ and  $\mu(f)$ and
$\nu(g)$ depend, respectively, on the restrictions $f|s(\mu)$ and
$g|s^*(\nu)$ for all $f\in C(X)$ and $g\in C^*(X)$. This implies
that $j_X$ is one-to-one. Using again Lemma 2.1 and Proposition 2.2,
one can show that $j_X$ is surjective. According to next
proposition, $j_X$ is not always a homeomorphism.

A subset $A$ of a space $X$ is said to be {\em bounded} if
$f(A)\subset\mathbb R$ is bounded for every $f\in C(X)$. This notion
should be distinguished from the notion of a bounded set in a linear
topological space.

\begin{pro} For a given space $X$ we have:
\begin{itemize}
\item[(i)] The map $j_X$ is a homeomorphism if and only if $X$ is
pseudocompact;
\item[(ii)] $P_c(X)$ is metrizable if and only if $X$ is compact and
metrizable.
\end{itemize}
\end{pro}

\begin{proof} (i)
Obviously, if $X$ is pseudocompact, then $C(X)=C^*(X)$ and $j_X$ is
the identity on $P_c(X)$. Suppose $X$ is not pseudocompact and
choose $g\in C(X)$ and a discrete countable set $\{x(n):n\geq 1\}$
in $X$ such that $\{g(x(n)):n\geq 1\}$ is unbounded and discrete in
$\mathbb R$. For every $n\geq 2$ define the measures $\mu_n\in
P_c(X)$ and $\nu_n\in P_c^*(X)$ as follows: $\mu_1=\delta_{x(1)}$,
$\displaystyle\mu_n=(1-1/n)\delta_{x(1)}+\sum_{k=2}^{n+1}(1/n)^2\delta_{x(k)}$
and $\nu_1=\delta^*_{x(1)}$,
$\displaystyle\nu_n=(1-1/n)\delta^*_{x(1)}+\sum_{k=2}^{n+1}(1/n)^2\delta^*_{x(k)}$.
Obviously, $j_X(\mu_n)=\nu_n$ for all $n\geq 1$ and
$s(\mu_n)=s^*(\nu_n)=\{x(1),x(2),..,x(n+1)\}$, $n\geq 2$. So,
$g\big(\bigcup_{n=1}^{\infty}s(\mu_n)\big)$ is unbounded in $\mathbb
R$. This, according to \cite[Proposition 3.1]{vv} (see also
\cite{a2}), means that the sequence $\{\mu_n\}_{n\geq 1}$ is not
compact. On the other hand, it is easily seen that $\{\nu_n\}_{n\geq
2}$ converges in $P_c^*(X)$ to $\nu_1$. Consequently, $j_X$ is not a
homeomorphism.

(ii) First we prove that $P_c(\mathbb N)$ is not metrizable, where
$\mathbb N$ is the set of the integers $n\geq 1$ with the discrete
topology. For every $n\geq 1$ let $K(n)=P_c(\{1,2,..,n\})$.
Obviously, every $K(n)$ is homeomorphic to a simplex of dimension
$n-1$ and $K(n)\subset K(m)$ for $n\leq m$. Moreover, $P_c(\mathbb
N)=\bigcup_{n\geq 1}K(n)$.

\textit{Claim $2$. $P_c(\mathbb N)$ is nowhere locally compact}.

Indeed, otherwise there would be $\mu\in P_c(\mathbb N)$ and its
open neighborhood $O(\mu)$ in $P_c(\mathbb N)$ with
$\overline{O(\mu)}$ being compact. Then, by \cite[Proposition
3.1]{vv}, $S=\cup\{s(\nu):\nu\in O(\mu)\}$ is a bounded subset of
$\mathbb N$. Hence, $S\subset\{1,2,..,p\}$ for some $p\geq 1$. The
last inclusion means that $O(\mu)\subset K(p)$, so $\dim O(\mu)\leq
p-1$. Therefore, $O(\mu)$ being open in $P_c(\mathbb N)$ is also
open in each $K(n)$, $n>p$. Since every open subset of $K(n)$ is of
dimension $n-1$, we obtain that $\dim O(\mu)>p-1$, a contradiction.

Now, suppose $P_c(\mathbb N)$ is metrizable and fix $\mu\in
P_c(\mathbb N)$. Since $P_c(\mathbb N)$ is nowhere locally compact
and $K(n)$, $n\geq 1$, are compact, $U(\mu)-K(n)\neq\varnothing$ for
all $n\geq 1$ and all neighborhoods $U(\mu)\subset P_c(\mathbb N)$
of $\mu$. Using the last condition and the fact that $\mu$ has a
countable local base (as a point in a metrizable space), we can
construct a sequence $\{\mu_n\}_{n\geq 1}$ converging to $\mu$ in
$P_c(\mathbb N)$ such that $\mu_n\not\in K(n)$ for all $n$.
Consequently, $s(\mu_n)\nsubseteq\{1,2,..,n\}$, $n\geq 1$. To obtain
a contradiction, we apply again \cite[Proposition 3.1]{vv} to
conclude that $s(\mu)\cup\bigcup_{n\geq 1}s(\mu_n)$ is a bounded
subset of $\mathbb N$ because $\{\mu, \mu_n:n\geq 1\}$ is a compact
subset of $P_c(\mathbb N)$. Therefore, $P_c(\mathbb N)$ is not
metrizable.

Let us complete the proof of (ii). If $X$ is compact metrizable,
then $P_c(X)$ is metrizable (see, for example \cite{ff}). Suppose
$P_c(X)$ is metrizable. Then, by Proposition 2.3(i), $X$ is also
metrizable. If $X$ is not compact, it should contain a $C$-embedded
copy of $\mathbb N$ and, according to Proposition 2.3(ii), $P_c(X)$
should contain a copy of $P_c(\mathbb N)$. So, $P_c(\mathbb N)$
would be also metrizable, which is not possible. Therefore,  $X$ is
compact and metrizable provided $P_c(X)$ is metrizable.
\end{proof}

\begin{pro}
If one of the spaces $P_c(X)$ and $P_c^*(X)$ is \v{C}ech-complete,
then $X$ is pseudocompact.
\end{pro}

\begin{proof}
We prove first that non of the spaces $P_c(\mathbb N)$ and
$P_c^*(\mathbb N)$ is \v{C}ech-complete. Indeed, suppose
$P_c(\mathbb N)$ is \v{C}ech-complete. Since $P_c(\mathbb N)$ is
Lindel\"{o}f (as the union of the compact sets
$K(n)=P_c(\{1,2,..,n\})$), it is a p-paracompact in the sense of
Arhangel'skii \cite{a1}. So, there exists a perfect map $g$ from
$P_c(\mathbb N)$ onto a separable metric space $Z$. Then the
diagonal product $q=g\triangle j_{\mathbb N}\colon Z\times
P_c^*(\mathbb N)$ is perfect (because $g$ is perfect) and one-to-one
(because $j_{\mathbb N}$ is one-to-one). Thus, $q$ is a
homeomorphism. Since $P_c^*(\mathbb N)$ is second countable
\cite{ch1}, $Z\times P_c^*(\mathbb N)$ is metrizable. Consequently,
$P_c(\mathbb N)$ is metrizable, a contradiction (see Proposition
2.5(ii)).

Suppose now that $P_c^*(\mathbb N)$ is \v{C}ech-complete, so it is a
Polish space. Since $P_c^*(\mathbb N)$ is the union of the  compact
sets $K^*(n)=P_c^*(\{1,2,..,n\})$, $n\geq 1$, there exists $m>1$
such that $K^*(m)$ has a non-empty interior. Then
$K(m)=P_c(\{1,2,..,m\})$ has a non-empty interior in $P_c(\mathbb
N)$ because $K(m)=j_{\mathbb N}^{-1}(K^*(m))$. According to Claim 2,
this is again a contradiction.

If $X$ is not pseudocompact, there exists a function $g\in C(X)$ and
a discrete set $A=\{x_n:n\geq 1\}$ in $X$ such that $g(x_n)\neq
g(x_m)$ for $n\neq m$ and $g(A)$ is a discrete unbounded subset of
$\mathbb R$. Since $g(A)$ is $C$-embedded in $\mathbb R$, it follows
that $A$ is also $C$-embedded in $X$. So, $A$ is a $C$-embedded copy
of $\mathbb N$ in $X$. Then, by Proposition 2.3, $P_c(X)$ contains a
closed copy of $P_c(\mathbb N)$ and $P_c^*(X)$ contains a closed
copy of $P_c^*(\mathbb N)$. Since non of $P_c(\mathbb N)$ and
$P_c^*(\mathbb N)$ is \v{C}ech-complete, non of $P_c(X)$ and
$P_c^*(X)$ can be \v{C}ech-complete. This completes the proof.
\end{proof}

We say that an inverse system $\displaystyle
S=\{X_\alpha,p^\alpha_\beta, A\}$ is {\em factorizing} \cite{ch3} if
for every $h\in C(X)$, where $X$ is the limit space of $S$, there
exists $\alpha \in A$ and $h_\alpha\in C(X_\alpha)$ with
$h=h_\alpha\circ p_\alpha$. Here, $p_\alpha\colon X\to X_\alpha$ is
the $\alpha$-th limit projection. According to \cite{ch1}, $P_c^*$
is a continuous functor, i.e. for every factorizing inverse system
$S$ the space $P_c^*(\lim S)$ is the limit of the inverse system
$P_c^*(S)=\{P_c^*(X_\alpha),P_c^*(p^\alpha_\beta), A\}$. The same is
true for the functor $P_c$.

\begin{pro}
$P_c$ is a continuous functor.
\end{pro}

\begin{proof}
Let $\displaystyle S=\{X_\alpha,p^\alpha_\beta, A\}$ be a
factorizing inverse system with a limit space $X$ and let
$\{\mu_\alpha:\alpha\in A\}$ be a thread of the system $P_c(S)$. For
every $\alpha\in A$ we consider the measure
$\displaystyle\nu_\alpha=j_{X_\alpha}(\mu_\alpha)$. Here,
$\displaystyle j_{X_\alpha}\colon P_c(X_\alpha)\to P_c^*(X_\alpha)$
is the one-to-one surjection defined above. It is easily seen that
$\{\nu_\alpha:\alpha\in A\}$ is a thread of the system $P_c^*(S)$,
so it determines a unique measure $\nu\in P_c^*(X)$ (recall that
$P_c^*$ is a continuous functor). There exists a unique measure
$\mu\in P_c(X)$ with $j_X(\mu)=\nu$. One can show that
$P_c(p_\alpha)(\mu)=\mu_\alpha$ for all $\alpha$. Hence, the set
$P_c(X)$ coincides with the limit set of the system $P_c(S)$.  It
remains to show that for every $\mu^0\in P_c(X)$ and its
neighborhood $U$ in $P_c(X)$ there exists $\alpha\in A$ and a
neighborhood $V$ of $\mu_\alpha^0=P_c(p_\alpha)(\mu^0)$ in
$P_c(X_\alpha)$ such that $P_c(p_\alpha)^{-1}(V)\subset U$. We can
suppose that $U=\{\mu\in P_c(X): |\mu(h_i)-\mu^0(h_i)|<\epsilon,
i=1,2,..,k\}$ for some $\epsilon>0$ and $h_i\in C(X)$, $i=1,2,..,k$.
Since $S$ is factorizing, we can find $\alpha\in A$ and functions
$g_i\in C(X_\alpha)$ such that $h_i=g_i\circ p_\alpha$ for all
$i=1,..,k$. Then $V=\{\mu_\alpha\in P_c(X_\alpha):
|\mu_\alpha(g_i)-\mu^0_\alpha(g_i)|<\epsilon, i=1,2,..,k\}$ is the
required neighborhood of $\mu^0_\alpha$.
\end{proof}

%%%%%%%%%%%%%%%%%%%%%%%%%%%%%%
\section{Milyutin maps and linear operators with compact supports}

For every linear operator $u:C(X,E)\to C(Y,E)$, where $E$ is a
locally convex linear space, and $y\in Y$ there exists a linear map
$T(y)\colon C(X,E)\to E$ defined by $T(y)(g)=u(g)(y)$, $g\in
C(X,E)$. We say that $u$ has {\em compact supports} (resp., $u$ is
{\em regular}) if each $T(y)$ has a compact support in $X$ (resp.,
each $T(y)$ is regular). In a similar way we define a linear
operator with compact supports if $u\colon C(X,E)\to C^*(Y,E)$)
(resp., $u\colon C^*(X,E)\to C^*(Y,E)$ or $u\colon C^*(X,E)\to
C(Y,E)$)). Let us note that a linear map $u:C(X,E)\to C(Y,E)$
(resp., $u:C^*(X,E)\to C^*(Y,E)$) is regular and has compact
supports iff the formula

\medskip\noindent
(3)\hspace{0.3cm} $T(y)(g)=u(g)(y)$ with $g\in C(X,E)$ (resp., $g\in
C^*(X,E)$)

\medskip\noindent
produces a continuous map $T\colon Y\to P_c(X,E)$ (resp., $T\colon
Y\to P_c^*(X,E)$). If $f\colon X\to Y$ is a surjective map, then a
liner operator $u\colon C(X,E)\to C(Y,E)$ (resp., $u\colon
C^*(X,E)\to C^*(Y,E)$) is called {\em an averaging operator for $f$}
if $u(\varphi\circ f)=\varphi$ for every $\varphi\in C(Y,E)$ (resp.,
$\varphi\in C^*(Y,E)$). It is easily seen that $u\colon C(X,E)\to
C(Y,E)$ (resp., $u\colon C^*(X,E)\to C^*(Y,E)$) is a regular
averaging operator for $f$ with compact supports if and only if the
map $T\colon Y\to P_c(X,E)$ (resp., $T\colon Y\to P_c^*(X,E)$)
defined by (3), has the following property: the support of every
$T(y)$, $y\in Y$, is contained in $f^{-1}(y)$. Such a map $T$ will
be called {\em a map associated with $f$}. It is also clear that if
$T\colon Y\to P_c(X,E)$ (resp., $T\colon Y\to P_c^*(X,E)$) is a map
associated with $f$, then the equality (3) defines a regular
averaging operator $u\colon C(X,E)\to C(Y,E)$ (resp., $u\colon
C^*(X,E)\to C^*(Y,E)$) for $f$ with compact supports.

A surjective map $f\colon X\to Y$ is said to be {\em Milyutin} if
$f$ admits a regular averaging operator $u\colon C(X)\to C(Y)$ with
compact supports, or equivalently, there exists a map $T\colon Y\to
P_c(X)$  associated with $f$. A surjective map $f\colon X\to Y$ is
called {\em weakly Milyutin} (resp., {\em strongly Milyutin}) if
there exists a map $T\colon Y\to P_c^*(X)$ (resp., $T\colon
P_c(Y)\to P_c(X)$) such that $s^*\big(g(y)\big)\subset f^{-1}(y)$
for all $y\in Y$ (resp., $s\big(g(\mu)\big)\subset
f^{-1}\big(s(\mu)\big)$ for all $\mu\in P_c(Y)$). Obviously, every
strongly Milyutin map is Milyutin. Moreover, if $T\colon Y\to
P_c(X)$ is a map associated with $f$, then the map $j_X\circ T\colon
Y\to P_c^*(X)$ is witnessing that Milyutin maps are weakly Milyutin.
One can also show that if $f\colon X\to Y$ is weakly Milyutin, then
its \v{C}ech-Stone extension $\beta f\colon\beta X\to\beta Y$ is a
Milyutin map.

We are going to establish some properties of (weakly) Milyutin maps.

\begin{pro}
Let $f\colon X\to Y$ be a weakly Milyutin map and $E$ a complete
locally convex space. Then $f$ admits  a regular averaging operator
$u\colon C^*(X,E)\to C^*(Y,E)$ with compact supports.
\end{pro}

\begin{proof}
Let $T\colon Y\to P_c^*(X)$ be a map associated with $f$. For every
$g\in C^*(X,E)$ let $B(g)=\overline{conv~g(X)}$ and consider the map
$P_c^*(g)\colon P_c^*(X)\to P_c^*(B(g))$. Since $B(g)$ is a closed
and bounded in $E$ and $E$ is complete, by \cite[Theorem 3.4 and
Proposition 3.10]{b1}, there exists a continuous  map $b\colon
P_c^*(B(g))\to B(g)$ assigning to each measure its barycenter. The
composition $e(g)=b\circ P_c^*(g)\colon P_c^*(X)\to E$ is a
continuous extension of $g$ (we consider $X$ as a subset of
$P_c^*(X)$). Now, we define $u\colon C^*(X,E)\to C^*(Y,E)$ by
$u(g)=e(g)\circ T$. This a linear operator because
$e(g)(\mu)=\int_Xgd\mu$ for every $\mu\in P_c^*(X)$. Since $e(g)$ is
a map from $P_c^*(X)$ into $B(f)$, the linear map $\Lambda(y)\colon
C^*(X,E)\to E$, $\Lambda(y)(g)=u(g)(y)$, is regular for all $y\in
Y$.

So, it remains to show that the support of each $\Lambda(y)$ is
compact and it is contained in $f^{-1}(y)$. Because $T$ is
associated with $f$, $K(y)=s^*(T(y))$ is a compact subset of
$f^{-1}(y)$, $y\in Y$.
 We are going to show that if $h|K(y)=g|K(y)$ with $h,g\in
C^*(X,E)$, then $\Lambda(y)(h)=\Lambda(y)(g)$. That would imply the
support of $\Lambda(y)$ is contained in $K(y)\subset f^{-1}(y)$, and
hence it should be compact. To this end, observe that $T(y)$ can be
considered as an element of $P(K(y))$ - the probability measures on
$K(y)$. So, $T(y)$ is the limit of a net $\{\mu_\alpha\}\subset
P(K(y))$ consisting of measures with finite supports. Each
$\mu_\alpha$ is of the form
$\displaystyle\sum_{i=1}^{k(\alpha)}\lambda_i^\alpha\delta^*_{x_i^\alpha}$,
where $x_i^\alpha\in K(y)$ and $\lambda_i^\alpha$ are positive reals
with $\sum_{i=1}^{k(\alpha)}\lambda_i^\alpha=1$. Then
$\{e(g)(\mu_\alpha)\}$ converges to $e(g)(T(y))$ and
$\{e(h)(\mu_\alpha)\}$ converges to $e(h)(T(y))$. On the other hand,
$e(h)(\mu_\alpha)=\int_Xhd\mu_\alpha=\sum_{i=1}^{k(\alpha)}\lambda_i^\alpha
h(x_i^\alpha)$ and
$e(g)(\mu_\alpha)=\sum_{i=1}^{k(\alpha)}\lambda_i^\alpha
g(x_i^\alpha)$. Since $h|K(y)=g|K(y)$, $h(x_i^\alpha)=g(x_i^\alpha)$
for all $\alpha$ and $i$. Hence, $e(h)(T(y))=e(g)(T(y))$ which means
that $\Lambda(y)(h)=\Lambda(y)(g)$. Therefore, $u$ is a regular
averaging operator for $f$ and has compact supports.
\end{proof}

\begin{cor}
Let $X$ be a complete bounded convex subset of a locally convex
space and $f\colon X\to Y$ be a weakly Milyutin map such that
$f^{-1}(y)$ is convex for every $y\in Y$. Then there exists a map
$g\colon Y\to X$ such that $g(y)\in f^{-1}(y)$ for all $y\in Y$.
\end{cor}

\begin{proof}
Let $T\colon Y\to P_c^*(X)$ be a map associated with $f$. By
\cite[Proposition 3.10]{b1}, the barycenter $b(\mu)$ of each measure
$\mu\in P_c^*(X)$ belongs to $X$ and the map $b\colon P_c^*(X)\to X$
is continuous. Since the support of each $T(y)$, $y\in Y$, is
compact subset of $f^{-1}(y)$ and
$\overline{conv~s^*\big(T(y)\big)}\subset f^{-1}(y)$ (recall that
$f^{-1}(y)$ is convex), $b(T(y))\in f^{-1}(y)$. So, the map
$g=b\circ T$ is as required.
\end{proof}

Recall that a set-valued map $\Phi\colon X\to Y$ is lower
semi-continuous (br., lsc) if for every open $U\subset Y$ the set
$\Phi^{-1}(U)=\{x\in X: \Phi(x)\cap U\neq\varnothing\}$ is open in
$X$.

\begin{lem}
For every space $X$ and a linear space $E$ the set-valued map
$\Phi_X\colon P_c(X,E)\to X$, $($resp., $\Phi_X^*\colon
P_c^*(X,E)\to X$$)$ defined by $\Phi_X(\mu)=s(\mu)$,
$($resp.,$\Phi_X^*(\mu)=s^*(\mu)$$)$ is lsc.
\end{lem}

\begin{proof}
A similar statement was established in \cite[Lemma 1.2.7]{bd}, so we
omit the arguments.
\end{proof}

\begin{pro}
Let $f\colon X\to Y$ be a weakly Milyutin map. Then we have:
\begin{itemize}
\item[(i)] $\beta f\colon\beta X\to\beta Y$ is a Milyutin map;
\item[(ii)] $f$ is a Milyutin map provided $f$ is perfect.
\end{itemize}
\end{pro}

\begin{proof}
Let $T\colon Y\to P_c^*(X)$ be a map associated with $f$. To prove
(i), observe that $P_c^*(i):P_c^*(X)\to P_c(\beta X)$ is an
embedding, where $i:X\to\beta X$ is the standard embedding (see
Proposition 2.3(iii)). Because $P_c(\beta X)=P(\beta X)$ is compact,
we can extend $T$ to a map $\tilde{T}\colon\beta Y\to P(\beta X)$.
It suffices to show that $\tilde{T}$ is a map associated with $\beta
f$. To this end, consider the lsc map $\Phi=\beta f\circ\Phi_{\beta
X}\circ\tilde{T}\colon\beta Y\to\beta Y$. Since $\Phi$ is lsc and
$\Phi(y)=y$ for all $y\in Y$, $\Phi(y)=y$ for any $y\in\beta Y$.
This means that the support of any $\tilde{T}(y)$, $y\in\beta Y$, is
contained in $(\beta f)^{-1}(y)$. So, $\beta f$ is a Milyutin map.

The proof of (ii) follows from $(i)$ and the following result of
Choban \cite[Proposition 1.1]{co}: if $\beta f$ admits a regular
averaging operator and $f$ is perfect, then $f$ admits a regular
averaging operator $u\colon C(X)\to C(Y)$ such that
$$\inf\{h(x):x\in f^{-1}(y)\}\leq u(h)(y)\leq\sup\{h(x):x\in
f^{-1}(y)\}$$ for every $h\in C(X)$ and $y\in Y$. This implies that
the support of each linear map $T(y)\colon C(X)\to\mathbb R$, $y\in
Y$, defined by (3), is contained in $f^{-1}(y)$. Hence,
$s\big(T(y)\big)$ is compact because so is $f^{-1}(y)$ (recall that
$f$ is perfect). Therefore, $f$ is a Milyutin map.
\end{proof}

\begin{pro}
Let $f\colon X\to Y$ be a Milyutin map. Then, in each of the
following cases $f$ is strongly Milyutin: $($i$)$ $f^{-1}(K)$ is
compact for every compact set $K\subset Y$; $($ii$)$ every closed
and bounded subset of $X$ is compact.
\end{pro}

\begin{proof}
Let $u\colon C(X)\to C(Y)$, $u(h)(y)=g(y)(h)$, be a corresponding
regular averaging operator with compact supports, where $g\colon
Y\to P_c(X)$ is a map associated with $f$. We are going to extend
$g$ to a map $\tilde{g}\colon P_c(Y)\to P_c(X)$ such that
$s\big(\tilde{g}(\mu)\big)\subset f^{-1}\big(s(\mu)\big)$ for all
$\mu\in P_c(Y)$). Let $\mu\in P_c(Y)$ and $K=s(\mu)\subset Y$. Then
$g(K)$ is a compact subset of $P_c(X)$. Hence, by \cite[Proposition
3.1]{vv}, $H=\overline{\cup\{s\big(g(y)\big):y\in K\}}$ is a bounded
and closed subset of $X$. Since $s\big(g(y)\big)\subset f^{-1}(y)$
for all $y\in Y$, $H\subset f^{-1}(K)$. So, in each of the cases (i)
and (ii), $H$ is compact. Define $\tilde{g}(\mu)\colon
C(X)\to\mathbb R$ to be the linear functional
$\tilde{g}(\mu)(h)=\mu(u(h))$, $h\in C(X)$. One can check that
$\tilde{g}(\mu)(h)=0$ provided $h(H)=0$. This means that the support
of $\tilde{g}(\mu)$ is a compact subset of $H$, so
$\tilde{g}(\mu)\in P_c(X)$. Moreover, $\tilde{g}$, considered as a
map from $P_c(Y)$ to $P_c(X)$ is continuous and satisfies the
inclusions $s\big(\tilde{g}(\mu)\big)\subset
f^{-1}\big(s(\mu)\big)$, $\mu\in P_c(Y)$. Therefore, $f$ is strongly
Milyutin.
\end{proof}

A map $f\colon X\to Y$ is said to be {\em $0$-invertible} \cite{hf}
if for any space $Z$ with $\dim Z=0$ and any map $p\colon Z\to Y$
there exists a map $q\colon Z\to X$ such that $f\circ q=p$. Here,
$\dim Z=0$ means that $\dim\beta Z=0$. We say that $f\colon X\to Y$
has {\em a metrizable kernel} if there exists a metrizable space $M$
and an embedding $X\subset Y\times M$ such that $\pi_Y|X=f$, where
$\pi_Y\colon Y\times M\to Y$ is the projection.

Next theorem is a generalization of \cite[Theorem 3.4]{dt} and
\cite[Corollary 1]{hf}.

\begin{thm}
Let $f\colon X\to Y$ be a surjection with a metrizable kernel and
$Y$ a paracompact space. Then the following conditions are
equivalent:
\begin{itemize}
\item[(i)] $f$ is $($weakly$)$ Milyutin;
\item[(ii)] The set-valued map $f^{-1}\colon Y\to X$ admits a lsc
 compact-valued selection;
\item[(iii)] $f$ is $0$-invertible.
\end{itemize}
\end{thm}

\begin{proof}
$(i)\Rightarrow (ii)$ Let $f$ be weakly Milyutin and $T\colon Y\to
P_c^*(X)$ is a map associated with $f$. By Lemma 3.3, the map
$\Phi_X^*\colon P_c^*(X)\to X$ is lsc, so is the map $\Phi_X^*\circ
T$. Moreover, $\Phi_X^*\big(T(y)\big)=s^*\big(T(y)\big)\subset
f^{-1}(y)$ for all $y\in Y$. Hence, $\Phi_X^*\circ T$ is a
compact-valued selection of $f^{-1}$.

$(ii)\Rightarrow (iii)$ Suppose $M$ is a metrizable space such that
$X\subset Y\times M$ and $\pi_Y|X=f$. Suppose also that $f^{-1}$
admits a compact-valued lsc selection $\Phi\colon Y\to X$. To show
that $f$ is 0-invertible, take a map $p\colon Z\to Y$  with $\dim
Z=0$, and let $Z_1=(\beta p)^{-1}(Y)$. Then $Z_1$ is paracompact (as
a perfect preimage of $Y$) and $\dim Z_1=0$ because $\beta Z_1=\beta
Z$ is 0-dimensional. The set-valued map $\pi_M\circ\Phi\circ
p_1\colon Z_1\to M$ is lsc and compact-valued, where $\pi_M\colon
Y\times M\to M$ is the projection and $p_1=(\beta p)|Z_1$. According
to \cite{m1}, $\pi_M\circ\Phi\circ p_1$ admits a (single-valued)
continuous selection $q_1\colon Z_1\to M$. Finally, the map $q\colon
Z\to X$, $q(z)=\big(p(z),q_1(z)\big)$ is the required lifting of
$p$, i.e. $f\circ q=p$.

$(iii)\Rightarrow (i)$ By \cite{rss}, there exists a perfect weakly
Milyutin map $p\colon Z\to Y$ with $Z$ being a 0-dimensional
paracompact. Then, by Proposition 3.4(ii), $p$ is a Milyutin map.
Since $f$ is 0-invertible, there exists a map $g\colon Z\to X$ with
$f\circ g=p$. If $T\colon Y\to P_c(Z)$ is a map associated with $p$,
then $\tilde{T}=P_c(g)\circ T\colon Y\to P_c(X)$ is a map associated
with $f$ because $s\big(\tilde{T}(y)\big)\subset
g\big(p^{-1}(y)\big)\subset f^{-1}(y)$ for all $y\in Y$. Hence, $f$
is a Milyutin map.
\end{proof}

\begin{cor}
Let $f\colon X\to Y$ be a surjective map such that either $X$ and
$Y$ are metrizable or $f$ is perfect. Then the following are
equivalent: $(i)$ $f$ is weakly Milyutin; $(ii)$ $f$ is Milyutin;
$(iii)$ $f$ is strongly Milyutin.
\end{cor}

\begin{proof}
If $X$ and $Y$ are metrizable, this follows from Proposition 3.5 and
Theorem 3.6. In case $f$ is perfect, we apply Propositions 3.4 and
3.5.
\end{proof}

A space $Z$ is called {\em a $k_{\mathbb R}$-space} if every
function on $Z$ is continuous provided it is continuous on every
compact subset of $Z$.

\begin{thm}
The product $f$ of any family $\{f_\alpha\colon X_\alpha\to
Y_\alpha, \alpha\in A\}$  of weakly Milyutin maps is also weakly
Milyutin. If, in addition, $Y=\prod\{Y_\alpha:\alpha\in A\}$ is a
$k_{\mathbb R}$-space and for every $\alpha\in A$ the closed and
bounded subsets of $X_\alpha$ are compact, then $f$ is Milyutin
provided each $f_\alpha$ is Milyutin.
\end{thm}

\begin{proof}
Let $T_\alpha:Y_\alpha\to P_c^*(X_\alpha)$ be a map associated with
$f_\alpha$ for each $\alpha$. Then, by Proposition 3.4, $\beta
f_\alpha$ is a Milyutin map and $\beta T_\alpha\colon\beta
Y_\alpha\to P(\beta X_\alpha)$ is associated with $\beta f_\alpha$.
So, $u_\alpha\colon C(\beta X_\alpha)\to C(\beta Y_\alpha)$,
$u_\alpha(h)(y)=\beta T_\alpha(y)(h)$, $y\in\beta Y_\alpha$ and
$h\in C(\beta X_\alpha)$, is a regular averaging operator for $\beta
f_\alpha$. Let $X=\prod\{X_\alpha:\alpha\in A\}$,
$\tilde{X}=\prod\{\beta X_\alpha:\alpha\in A\}$,
$\tilde{Y}=\prod\{\beta Y_\alpha:\alpha\in A\}$ and
$\tilde{f}=\prod\{\beta f_\alpha:\alpha\in A\}$. According to
\cite{p}, there exists a regular averaging operator
$u:C(\tilde{X})\to C(\tilde{Y})$ for $\tilde{f}$ such that $u(h\circ
p_\alpha)=u_\alpha(h)\circ q_\alpha$, $\alpha\in A$, $h\in C(\beta
X_\alpha)$, where $p_\alpha:\tilde{X}\to\beta X_\alpha$ and
$q_\alpha:\tilde{Y}\to\beta Y_\alpha$ are the projections. This
implies that, if $\tilde{T}\colon\tilde{Y}\to P(\tilde{X})$ is the
map associated to $\tilde{f}$ and generated by $u$, we have
$s\big(\tilde{T}(y)\big)\subset\prod\{s\big(T_\alpha(q_\alpha(y))\big):\alpha\in
A\}$, $y\in Y$. Hence, $s\big(\tilde{T}(y)\big)\subset f^{-1}(y)$
for every $y\in Y$. So, $\tilde{T}$ maps $Y$ into the subspace $H$
of $P(\tilde{X})$ consisting of all measures $\mu\in P(\tilde{X})$
with $s(\mu)\subset X$. Now, let $\pi\colon\beta X\to\tilde{X}$ be
the natural map and $P(\pi)\colon P(\beta X)\to P(\tilde{X})$. Then,
$\theta=P(\pi)|P_c^*(X)\colon P_c^*(X)\to H$ is a homeomorphism (for
more general result see \cite[Proposition 1]{ch1}). Therefore,
$T=\theta^{-1}\circ (\tilde{T}|Y)\colon Y\to P_c^*(X)$ is a map
associated with $f$. Thus, $f$ is weakly Milyutin.

Suppose now that $Y$ is a $k_{\mathbb R}$-space, $f_\alpha$ are
Milyutin maps and the closed and bounded subsets of each $X_\alpha$
are compact. We already proved that there exists a regular averaging
operator $u\colon C^*(X)\to C^*(Y)$ for $f$ and a corresponding to
$u$ map $T\colon Y\to P_c^*(X)$ associated with $f$ such that
$s^*(T(y))\subset\prod\{s(T_\alpha(q_\alpha(y))):\alpha\in
A\}\subset f^{-1}(y)$ for every $y\in Y$. Here, each $T_\alpha\colon
Y_\alpha\to P_c(X_\alpha)$ is a map associated with $f_\alpha$
(recall that $f_\alpha$ are Milyutin maps). For any $h\in C(X)$ and
$n\geq 1$ define $h_n\in C^*(X)$ by $h_n(x)=h(x)$ if $|h(x)|\leq n$,
$h_n(x)=n$ if $h(x)\geq n$ and $h_n(x)=-n$ if $h(x)\leq -n$. Since
for every $y\in Y$ the support $s^*(T(y))\subset X$ is compact,
$h|s^*(T(y))=h_n|s^*(T(y))$ with $n\geq n_0$ for some $n_0$. Hence,
the formula $v(h)(y)=\lim u(h_n)(y)$, $y\in Y$, defines a function
on $Y$. Let us show that $v(h)$ is continuous. Since $Y$ is a
$k_{\mathbb R}$-space, it suffices to prove that $v(h)$ is
continuous on every compact set $K\subset Y$. Then each of the sets
$T_\alpha(K_\alpha)\subset P_c(X_\alpha)$ is compact, where
$K_\alpha=q_\alpha(K)$.  By \cite[Proposition 3.1]{vv},
$Z_\alpha=\overline{\cup\{s(\mu):\mu\in T_\alpha(K_\alpha)\}}$ is
bounded in $X_\alpha$ and, hence compact (recall that all closed and
bounded subsets of $X_\alpha$ are compact). Let $Z$ be the closure
in $X$ of the set $\cup\{s^*(\mu):\mu\in T(K)\}$. Since
$Z\subset\prod\{Z_\alpha:\alpha\in A\}$, $Z$ is also compact. So,
there exists $m$ such that $h|Z=h_n|Z$ for all $n\geq m$. This
implies that $v(h)|K=u(h_m)|K$. Hence, $v(h)$ is continuous on $K$.
Since for every $y\in Y$ the support of $T(y)$ is  compact and each
$u(h)(y)$, $h\in C^*(X)$, depends on $h|s^*(T(y))$, $v\colon C(X)\to
C(Y)$ is linear and the support of $T'(y)\in P_c(X)$ is contained in
$s^*(T(y))\subset f^{-1}(y)$, where $T'\colon Y\to P_c(X)$ is
defined by $T'(y)(h)=v(h)(y)$, $h\in C(X)$, $y\in Y$. Moreover, it
follows from the definition of $v$ that it is regular and
$v(\phi\circ f)=\phi$ for every $\phi\in C(Y)$. Therefore, $v$ is a
regular averaging operator for $f$ with compact supports
\end{proof}

\begin{cor}
A product of perfect Milyutin maps is also Milyutin.
\end{cor}

\begin{proof}
Since any product of perfect maps is perfect, the proof follows from
Corollary 3.7 and Theorem 3.8.
\end{proof}

\begin{cor}
Let $Y=\prod\{Y_\alpha:\alpha\in A\}$ be a product of metrizable
spaces. Then there exists a $0$-dimensional product $X$ of
metrizable spaces space and a $0$-invertible perfect Milyutin map
$f\colon X\to Y$.
\end{cor}

\begin{proof}
By \cite[Theorem 1.2.1]{co}, for every $\alpha\in A$ there exists a
0-dimensional metrizable space $X_\alpha$ and a perfect Milyutin map
$f_\alpha\colon X_\alpha \to Y_\alpha$. Then, by Corollary 3.9,
$f=\prod\{f_\alpha:\alpha\in A\}$ is a perfect Milyutin map from
$X=\prod\{X_\alpha:\alpha\in A\}$ onto $Y$. It is easily seen that
$f$ is $0$-invertible because each $f_\alpha$ is $0$-invertible (see
Theorem 3.6). Moreover, since $\dim X_\alpha=0$ for each $\alpha$,
$\dim X=0$.
\end{proof}

Recall that $X$ is a $p$-paracompact space \cite{a1} if it admits a
perfect map onto a metrizable space.

\begin{cor}
For every $p$-paracompact space $Y$ there exists a $0$-dimensional
$p$-paracompact space $Y$ and a perfect $0$-invertible Milyutin map
$f\colon X\to Y$.
\end{cor}

\begin{proof}
Since $Y$ is $p$-paracompact, it can be considered as a closed
subset of $M\times\mathbb I^\tau$,where $M$ is metrizable and
$\tau\geq\aleph_0$. There exist perfect Milyutin maps
$\phi:\mathfrak C\to\mathbb I$ and $g\colon M_0\to M$ with
$\mathfrak C$ being the Cantor set \cite{p} and $M_0$ a
$0$-dimensional metrizable space. \cite[Theorem 1.2.1]{co}. Then the
product map $\Phi=g\times\phi^\tau\colon M_0\times\mathfrak C^\tau$
is a perfect $0$-invertible Milyutin map (see Corollary 3.10), and
let $T\colon M\times\mathbb I^\tau\to P_c\big(M_0\times\mathfrak
C^\tau\big)$ be a map associated with $\Phi$. Define
$X=\Phi^{-1}(Y)$ and $f=\Phi|X$. Since $X$ is closed in
$M_0\times\mathfrak C^\tau$, it is a $0$-dimensional
$p$-paracompact. Since $\Phi$ is $0$-invertible (as a product of
$0$-invertible maps, see Theorem 3.6), so is $f$. To show that $f$
is Milyutin, observe that $X$ is $C$-embedded in $M_0\times\mathfrak
C^\tau$. So, $P_c(X)$ is embedded in $P_c\big(M_0\times\mathfrak
C^\tau\big)$ such that $T(y)\in P_c(X)$ for all $y\in Y$. This means
that $T|Y$ is a map associated with $f$. Hence, $f$ is Milyutin.
\end{proof}

Now, we provide a specific class of Milyutin maps. Suppose $B\subset
Z$ and $g\colon B\to D$. We say that $g$ is {\em a $Z$-normal map}
provided for every $h\in C(D)$ the function $h\circ g$ can be
continuously extended to a function on $Z$. A map $f:X\to Y$ is
called {\em 0-soft} \cite{ch2} if for any 0-dimensional space $Z$,
any two subspaces $Z_0\subset Z_1\subset Z$, and any $Z$-normal maps
$g_0:Z_0\to X$ and $g_1\colon Z_1\to Y$ with $f\circ g_0=g_1|Z_0$,
there exists a $Z$-normal map $g\colon Z_1\to X$ such that $f\circ
g=g_1$.

\begin{pro}
Every $0$-soft map is Milyutin.
\end{pro}

\begin{proof}
Let $f\colon X\to Y$ be 0-soft. Consider $Y$ as a $C$-embedded
subset of $\mathbb R^{C(Y)}$ and let $\varphi\colon Z\to\mathbb
R^{C(Y)}$ be a perfect Milyutin map with $\dim Z=0$ (see Corollary
3.10).  Since $Y$ is $C$-embedded in $\mathbb R^{C(Y)}$,
$g_1=\varphi|Z_1:Z_1\to Y$ is a $Z$-normal map, where
$Z_1=\varphi^{-1}(Y)$. Because $f$ is 0-soft, there exists a
$Z$-normal map $g\colon Z_1\to X$ with $f\circ g=g_1$. Now, for
every $h\in C(X)$ choose an extension $e(h)\in C(Z)$ of $h\circ g$
(such $e(h)$ exist since $g$ is $Z$-normal). Define $v\colon C(X)\to
C(Y)$ by $v(h)=u(e(h))|Y$, where $u\colon C(Z)\to C\big(\mathbb
R^{C(Y)}\big)$ is a regular averaging operator for $\varphi$ having
compact supports. The map $v$ is linear because for every $y\in Y$
$u(e(h))(y)$ depends on the restriction $e(h)|\varphi^{-1}(y)$. By
the same reason $v$ has compact supports. Moreover, $v$ is a regular
averaging operator for $f$. Hence, $f$ is Milyutin.
\end{proof}
%%%%%%%%%%%%%%%%%%%%%%%%%%%%%%%%%%%%%%%%%%%%%%%%%%%%%%%%%%%%%%%%%%%%%%%%%%%%%%%%%%%%%%%%%%%%%%%%%%%%

%%%%%%%%%%%%%%%%%%%%%%%%%%%%%%
\section{AE(0)-spaces and regular extension operators with compact supports}

Let $X$ be a subspace of $Y$. A linear operator $u:C(X,E)\to C(Y,E)$
is said to be {\em an extension operator} provided each $u(f)$,
$f\in C(X,E)$ is an extension of $f$. One can show that such an
extension operator $u$ is regular and has compact supports if and
only if there exists a map $T\colon Y\to P_c(X,E)$ such that
$T(x)=\delta_x$ for every $x\in X$. Sometimes a map $T\colon Y\to
P_c(X,E)$ satisfying the last condition will be called {\em a
$P_c$-valued retraction}. The connection between $u$ and $T$ is
given by the formula $T(y)(f)=u(f)(y)$, $f\in C(X,E)$, $y\in Y$.

Pelczynski \cite{p} introduced the class of Dugundji spaces: a
compactum $X$ is a {\em Dugundji space} if for every embedding of
$X$ in another compact space $Y$ there exists an extension regular
operator $u\colon C(X)\to C(Y)$ (note that $u$ has compact supports
because $X$ is compact). Later Haydon \cite{ha} proved that a
compact space $X$ is a Dugundji space if and only if it is an
absolute extensor for $0$-dimensional compact spaces (br., $X\in
AE(0)$). The notion of $X\in AE(0)$ was extended by Chigogidze
\cite{ch2} in the class of all Tychonoff spaces as follows: a {\em
space $X$ is an $AE(0)$} if for every $0$-dimensional space $Z$ and
its subspace $Z_0\subset Z$, every $Z$-normal map $g\colon Z_0\to X$
can be extended to the whole of $Z$.

We show that an analogue of Haydon's result remains true and for the
extended class of $AE(0)$-spaces.

\begin{thm}\label{main}
For any space $X$ the following conditions are equivalent:
\begin{itemize}
\item[(i)] $X$ is an $AE(0)$-space;
\item[(ii)] For every $C$-embedding of $X$ in a space $Y$ there
exists a regular extension operator $u\colon C(X)\to C(Y)$ with
compact supports;
\item[(iii)] For every $C$-embedding of $X$ in a space $Y$ there
exists a regular extension operator $u\colon C^*(X)\to C^*(Y)$ with
compact supports.
\end{itemize}
\end{thm}

\begin{proof}
$(i)\Rightarrow (ii)$ Suppose $X$ is $C$-embedded in $Y$ and take a
set $A$ such that $Y$ is $C$-embedded in $\mathbb R^A$. It suffices
to show there exists a regular extension operator $u\colon C(X)\to
C(\mathbb R^A)$ with compact supports, or equivalently, we can find
a map $T\colon\mathbb R^A\to P_c(X)$ with $T(x)=\delta_x$ for all
$x\in X$. By Corollary 3.10, there exists a $0$-dimensional space
$Z$ and a Milyutin map $f\colon Z\to\mathbb R^A$. This means that
the map $g\colon\mathbb R^A\to P_c(Z)$ associated with $f$ is an
embedding.
 Since $X$ is $C$-embedded in $\mathbb R^A$, the restriction
$f|f^{-1}(X)$ is a $Z$-normal map. So, there exists a map $q\colon
Z\to X$ extending $f|f^{-1}(X)$ (recall that $X\in AE(0)$). Then
$T=P_c(q)\circ g\colon\mathbb R^A\to P_c(X)$ has the required
property that $T(x)=\delta_x$ for all $x\in X$.

$(ii)\Rightarrow (iii)$ Let $X$ be $C$-embedded in $Y$ and $u\colon
C(X)\to C(Y)$ a regular extension operator with compact supports.
Then $u(f)\in C^*(Y)$ for all $f\in C^*(X)$ because $u$ is regular.
Hence, $u|C^*(X)\colon C^*(X)\to C^*(Y)$ is a regular extension
operator with compact supports.

$(iii)\Rightarrow (i)$ Suppose $X$ is $C$-embedded in $\mathbb R^A$
for some $A$ and $u\colon C^*(X)\to C^*(\mathbb R^A)$ is a regular
extension operator with compact supports. So, there exists a map
$T\colon\mathbb R^A\to P_c(X)$ with $T(x)=\delta_x$, $x\in X$.
Assume that $A$ is the set of all ordinals
$\{\lambda:\lambda<\omega(\tau)\}$, where $\omega(\tau)$ is the
first ordinal of cardinality $\tau$.

For any sets $B\subset D\subset A$ we use the following notations:
$\pi_B\colon\mathbb R^A\to\mathbb R^B$ and $\pi^D_B\colon\mathbb
R^D\to\mathbb R^B$ are the natural projections, $X(B)=\pi_B(X)$,
$p_B=\pi_B|X$ and $p^D_B=\pi^D_B|X(D)$. A set $B\subset A$ is called
{\em $T$-admissible} if for any $x\in X$ and $y\in\mathbb R^A$ the
equality $\pi_B(x)=\pi_B(y)$ implies
$P^*_c(p_B)(\delta_x)=P^*_c(p_B)(T(y))$. Let us note that if $B$ is
$T$-admissible, then there exists a map

\medskip\noindent
(4)\hspace{0.3cm} $T_B\colon\mathbb R^B\to P_c^*(X(B))$ such that
$T_B(z)=\delta_z$ for all $z\in X(B)$.

\noindent Indeed, take an embedding $i\colon\mathbb R^B\to\mathbb
R^A$ such that $\pi_B\circ i$ is the identity on $\mathbb R^B$, and
define $T_B=P_c^*(p_B)\circ T\circ i$.

\medskip\noindent
{\em Claim $3$. For every countable set $B\subset A$ there exists a
countable $T$-admissible set $D\subset A$ containing $B$}

We construct by induction an increasing sequence $\{D(n)\}_{n\geq
1}$ of countable subsets of $A$ such that $D\subset D(1)$ and for
all $n\geq 1$, $x\in X$ and $y\in\mathbb R^A$ we have

\medskip\noindent
(5)\hspace{0.3cm} $P_c^*(p_{D(n)})(\delta_x)=P_c^*(p_{D(n)})(T(y))$
provided $\pi_{D(n+1)}(x)=\pi_{D(n+1)}(y)$.

\medskip\noindent
Suppose we have already constructed $D(1),..,D(n)$. Since $D(n)$ is
countable, the topological weight of $X(D(n))$ is $\aleph_0$. So is
the weight of $P_c^*(X(D(n)))$ \cite{ch1}. Then the map
$P_c^*(p_{D(n)})\circ T\colon\mathbb R^A\to P_c^*(X(D(n)))$ depends
on countable many coordinates (see, for example \cite{pp}). This
means that there exists a countable set $D(n+1)$ satisfying $(5)$.
We can assume that $D(n+1)$ contains $D(n)$, which completes the
induction. Obviously, the set $D=\bigcup_{n\geq 1}D(n)$ is
countable. Let us show it is $T$-admissible. Suppose
$\pi_D(x)=\pi_D(y)$ for some $x\in X$ and $y\in\mathbb R^A$. Hence,
for every $n\geq 1$ we have $\pi_{D(n+1)}(x)=\pi_{D(n+1)}(y)$ and,
by (5), $P_c^*(p_{D(n)})(\delta_x)=P_c^*(p_{D(n)})(T(y))$. This
means that the support of each measure $P_c^*(p_{D(n)})(T(y))$ is
the point $p_{D(n)}(x)$. The last relation implies that the support
of $P_c^*(p_D)(T(y))$ is the point $p_D(x)$. Therefore,
$P_c^*(p_D)(T(y))=P_c^*(p_D)(\delta_x)$ and $D$ is $T$-admissible.

\medskip\noindent
{\em Claim $4$. Any union of $T$-admissible sets is $T$-admissible}.

Suppose $B$ is the union of $T$-admissible sets $B(s)$, $s\in S$,
and $\pi_B(x)=\pi_B(y)$ with $x\in X$ and $y\in\mathbb R^A$. Then
$\pi_{B(s)}(x)=\pi_{B(s)}(y)$ for every $s\in S$. Hence,
$P_c^*(p_{B(s)})(T(y))=P_c^*(p_{B(s)})(\delta_x)$, $s\in S$. So, the
support of each $P_c^*(p_{B(s)})(T(y))$ is the point $p_{B(s)})(x)$.
Consequently, the support of $P_c^*(p_B)(T(y))$ is the point
$p_B(x)$ because
$p_B(x)=\bigcap\{\big(p^B_{B(s)}\big)^{-1}(p_{B(s)}(x)):s\in S\}$.
This means that $B$ is $T$-admissible.

\medskip\noindent
{\em Claim $5$. Let $B\subset A$ be $T$-admissible. Then we have:
\begin{itemize}
\item[(a)] $X(B)$ is a closed subset of $\mathbb R^B$;
\item[$(b)$] $p_B(V)$ is functionally open in $X(B)$ for any functionally
open subset $V$ of $X$.
\end{itemize}}

Since $B$ is $T$-admissible, according to (4) there exists a map
$T_B\colon\mathbb R^B\to P_c^*(X(B))$ such that $T_B(z)=\delta_z$
for all $z\in X(B)$. To prove condition (a), suppose
$\{z_\alpha:\alpha\in\Lambda\}$ is a net in $X(B)$ converging to
some $z\in\mathbb R^B$. Then $\{T_B(z_\alpha)\}$ converges to
$T_B(z)$. But $T_B(z_\alpha)=\delta_{z_\alpha}\in i^*_{X(B)}(X(B))$
for every $\alpha$ and, since $i^*_{X(B)}(X(B))$ is a closed subset
of $P_c^*(X(B))$ (see Proposition 2.3(i)), $T_B(z)\in
i^*_{X(B)}(X(B))$. Hence, $T_B(z)=\delta_y$ for some $y\in X(B)$.
Using that $i^*_{X(B)}$ embeds $X(B)$ in $P_c^*(X(B))$, we obtain
that $\{z_\alpha\}$ converges to $y$, so $y=z\in X(B)$.

To prove $(b)$, let $V$ be a functionally open subset of $X$ and
$g\colon X\to [0,1]$ a continuous function with $V=g^{-1}((0,1])$.
Then $u(g)\in C^*(\mathbb R^A)$ with $0\leq u(g)(y)\leq 1$ for all
$y\in\mathbb R^A$ and let $W=u(g)^{-1}((0,1])$. Since $\pi_B(W)$ is
functionally open in $\mathbb R^B$ (see, for example \cite{vv2}),
$\pi_B(W)\cap X(B)$ is functionally open in $X(B)$. So, it suffices
to show that $p_B(V)=\pi_B(W)\cap X(B)$. Because $u(g)$ extends $g$,
we have $V\subset W$. So, $p_B(V)\subset\pi_B(W)\cap X(B)$. To prove
the other inclusion, let $z\in\pi_B(W)\cap X(B)$. Choose $x\in X$
and $y\in W$ with $\pi_B(x)=\pi_B(y)$. Then
$P_c^*(p_B)(T(y))=P_c^*(p_B)(\delta_x)=\delta_z$ (recall that $B$ is
$T$-admissible). Hence, $s^*(T(y))\subset p_B^{-1}(z)$. Since $y\in
W$, $T(y)(g)=u(g)(y)\in (0,1]$. This implies that $s^*(T(y))\cap
V\neq\varnothing$ (otherwise $T(y)(g)=0$ because $g(X-V)=0$, see
Proposition 2.1(ii)). Therefore, $z\in p_B(V)$, i.e. $\pi_B(W)\cap
X(B)\subset p_B(V)$. The proof of Claim 5 is completed.

\medskip
Let us continue the proof of $(iii)\Rightarrow (i)$. Since $A$ is
the set of all ordinals $\lambda<\omega(\tau)$, according to Claim
3, for every $\lambda$ there exists a countable $T$-admissible set
$B(\lambda)\subset A$ containing $\lambda$. Let
$A(\lambda)=\cup\{B(\eta):\eta<\lambda\}$ if $\lambda$ is a limit
ordinal, and $A(\lambda)=\cup\{B(\eta):\eta\leq\lambda\}$ otherwise.
By Claim 4, every $A(\lambda)$ is $T$-admissible. We are going to
use the following simplified notations:

\begin{center}
$X_\lambda=X(A(\lambda))$, $p_\lambda=p_{A(\lambda)}\colon
X\to X_\lambda$ and $p^\eta_\lambda\colon X\eta\to X_\lambda$
provided $\lambda<\eta$.
\end{center}

\noindent Since $A$ is the union of all $A(\lambda)$ and each
$X_\lambda$ is closed in $\mathbb R^{A(\lambda)}$ (see Claim 5(a)),
we obtain a continuous inverse system $S=\{X_\lambda,
p^\eta_\lambda, \lambda<\eta<\omega(\tau)\}$ whose limit space is
$X$. Recall that $S$ is continuous if for every limit ordinal
$\gamma$ the space $X_\gamma$ is the limit of the inverse system
$\{X_\lambda, p^\eta_\lambda, \lambda<\eta<\gamma\}$. Because of the
continuity of $S$, $X\in AE(0)$ provided $X_1\in AE(0)$ and each
short projection $p^{\lambda+1}_\lambda$ is 0-soft. The space $X_1$
being a closed subset of $\mathbb R^{A(1)}$ is a Polish space, so an
$AE(0)$ \cite{ch2}. Hence, it remains to show that all
$p^{\lambda+1}_\lambda$ are 0-soft.

\medskip
We fix $\lambda<\omega(\tau)$ and let
$E(\lambda)=A(\lambda)\cap\big(B(\lambda)\cup B(\lambda+1)\big)$.
Since $E(\lambda)$ is countable, there exists a sequence
$\{\beta_n\}\subset A(\lambda)$ such that $\beta_n\leq\lambda$ for
each $n$ and $E(\lambda)\subset C(\lambda)\subset A(\lambda)$, where
$C(\lambda)=\cup\{B(\beta_n):n\geq 1\}$. By Claim 4, the sets
$C(\lambda)$ and $D(\lambda)=B(\lambda)\cup B(\lambda+1)\cup
C(\lambda)$ are countable and $T$-admissible. Consider the following
diagram:
$$
\begin{CD}
X_{\lambda+1}@>{p^{\lambda+1}_\lambda}>>X_\lambda\cr
@V{p^{A(\lambda+1)}_{D(\lambda)}}VV
@VV{p^{A(\lambda)}_{C(\lambda)}}V\cr
X(D(\lambda))@>{p^{D(\lambda)}_{C(\lambda)}}>>X(C(\lambda))\cr
\end{CD}
$$

We are going to prove first that the diagram is a cartesian square.
This means that the map $g\colon X_{\lambda+1}\to Z$,
$g(x)=\big(p^{A(\lambda+1)}_{D(\lambda)}(x),p^{\lambda+1}_\lambda(x)\big)$,
is a homeomorphism. Here $Z=\{(x_1,x_2)\in X(D(\lambda))\times
X_\lambda:p^{D(\lambda)}_{C(\lambda)}(x_1)=p^{A(\lambda)}_{C(\lambda)}(x_2)\}$
is the fibered product of $X(D(\lambda))$ and $X_\lambda$ with
respect to the maps $p^{D(\lambda)}_{C(\lambda)}$ and
$p^{A(\lambda)}_{C(\lambda)}$.
%
%Since $A(\lambda+1)$ and $D(\lambda)$ are $T$-admissible, it follows
%from Claim 5(b) that $p^{A(\lambda+1)}_{D(\lambda)}$ is a
%functionally open map, in particular open. Similarly,
%$p^{\lambda+1}_\lambda(x)$ is open. Therefore, $g$ is also open.
%
%To prove that $g$ is a homeomorphism, it suffices to show that $g$
%is surjective and one-to-one. To this end,
%
Let $z=\big(x(1),x(2)\big)\in Z$. Since
$\big(D(\lambda)-C(\lambda)\big)\cap\big(A(\lambda)-C(\lambda)\big)=\varnothing$
and
$A(\lambda+1)=\big(D(\lambda)-C(\lambda)\big)\cup\big(A(\lambda)-C(\lambda)\big)\cup
C(\lambda)$, there exists exactly one point $x\in\mathbb
R^{A(\lambda+1)}$ such that
$\pi^{A(\lambda+1)}_{D(\lambda)}(x)=x(1)$ and
$\pi^{A(\lambda+1)}_{A(\lambda)}(x)=x(2)$. Choose $y\in\mathbb R^A$
with $\pi_{A(\lambda+1)}(y)=x$. Since $D(\lambda)$ and $A(\lambda)$
are $T$-admissible, $P_c^*(p_{D(\lambda)})(T(y))=\delta_{x(1)}$ and
$P_c^*(p_{A(\lambda)})(T(y))=\delta_{x(2)}$. Consequently,
$p^{A(\lambda+1)}_{D(\lambda)}(H)=x(1)$ and
$p^{A(\lambda+1)}_{A(\lambda)}(H)=x(2)$, where $H$ is the support of
the measure $P_c^*(p_{A(\lambda+1)})(T(y))$. Hence, $H=\{x\}$ is the
unique point of $X_{\lambda+1}$ with $g(x)=z$. Thus, $g$ is a
surjective and one-to-one map between $X_{\lambda+1}$ and $Z$. To
prove $g$ is a homeomorphism, it remains to show that $g^{-1}$ is
continuous. The above arguments yield that $x=g^{-1}(z)$ depends
continuously from $z\in Z$. Indeed, since $D(\lambda)\cap
A(\lambda)=C(\lambda)$, we have
\begin{center}
$x(1)=(a,b)\in\mathbb R^{D(\lambda)-C(\lambda)}\times\mathbb
R^{C(\lambda)}$ and $x(2)=(b,c)\in\mathbb
R^{C(\lambda)}\times\mathbb R^{A(\lambda)-C(\lambda)}$,
\end{center}
\noindent where $z=(x(1),x(2))\in Z$. Hence, $g^{-1}(z)=(a,b,c)$ is
a continuous function of $z$.

Since $D(\lambda)$ and $C(\lambda)$ are countable and $T$-admissible
sets, both $X(D(\lambda))$ and $X(C(\lambda))$ are Polish spaces and
$p^{D(\lambda)}_{C(\lambda)}$ is functionally open (see Claim 5(b)).
Hence, $p^{D(\lambda)}_{C(\lambda)}$ is $0$-soft \cite{ch2}. This
yields that $p^{\lambda+1}_\lambda$ is also $0$-soft because the
above diagram is a cartesian square.
\end{proof}

Next proposition provides a characterization of $AE(0)$-spaces in
terms of extension of vector-valued functions. This result was
inspired by \cite{b2}.

\begin{thm}\label{vector-valued}
A space $X\in AE(0)$ if and only if for any complete locally convex
space $E$ and any $C$-embedding of $X$ in a space $Y$ there exists a
regular extension operator $\colon C^*(X,E)\to C^*(Y,E)$ with
compact supports.
\end{thm}

\begin{proof}
Suppose $X\in AE(0)$ and $X$ is $C$-embedded in a space $Y$. Then by
Theorem~\ref{main}(iii), there exists a regular extension operator
$v\colon C^*(X)\to C^*(Y)$ with compact supports. This is equivalent
to the existence of a $P_c^*$-valued retraction $T\colon Y\to
P_c^*(X)$.  We can extend each $f\in C^*(X,E)$ to a continuous
bounded map $e(f):P_c^*(X)\to E$. Indeed, let
$B(f)=\overline{conv~f(X)}$ and consider the map $P_c^*(f)\colon
P_c^*(X)\to P_c^*(B(f))$. Obviously, $B(f)$ is a bounded convex
closed subset $E$, so it is complete. Then, by \cite[Theorem 3.4 and
Proposition 3.10]{b1}, there exists a continuous map $b\colon
P_c^*(B(f))\to B(f)$ assigning to each measure $\nu\in P_c^*(B(f))$
its barycenter $b(\nu)$. The composition $e(f)=b\circ P_c^*(f)\colon
P_c^*(X)\to B(f)$ is a bounded continuous extension of $f$. We also
have

\medskip\noindent
(6)\hspace{0.3cm} $e(f)(\mu)=\int_Xfd\mu$ for every $\mu\in
P_c^*(X)$.

\medskip
\noindent Finally, we define $u\colon C^*(X,E)\to C^*(Y,E)$ by
$u(f)=e(f)\circ T$, $f\in C^*(X,E)$. The linearity of $u$ follows
from $(6)$. Moreover, for every $y\in Y$ the linear map
$\Lambda(y)\colon C^*(X,E)\to E$, $\Lambda(y)(f)=u(f)(y)$, is
regular because $\Lambda(y)(f)\in\overline{conv~f(X)}$. Using the
arguments from the proof of Proposition 3.1 (the final part), we can
show that each $\Lambda(y)$, $y\in Y$, has a compact support which
is contained in  $K(y)=s^*(T(y))\subset X$. Therefore, $u$ is a
regular extension operator with compact supports.

\medskip
The other implication follows from Theorem~\ref{main}. Indeed, since
$\mathbb R$ is complete, there exists a regular extension operator
$u\colon C^*(X)\to C^*(Y)$ provided $X$ is $C$-embedded in $Y$.
Hence, by Theorem~\ref{main}(iii), $X\in AE(0)$.
\end{proof}

Recall that a space $X$ is an absolute retract \cite{ch2} if for
every $C$-embedding of $X$ in a space $Y$ there exists a retraction
from $Y$ onto $X$.

\begin{cor}\label{AR}
Let $X$ be a convex bounded and complete subset of a locally convex
topological space. Then $X$ is an absolute retract provided $X\in
AE(0)$.
\end{cor}

\begin{proof}
Suppose $X$ is $C$-embedded in a space $Y$. According to
\cite[Theorem 3.4 and Proposition 3.10]{b1}, the barycenter of each
$\mu\in P_c(X)$ belongs to $X$ and the map $b\colon P_c(X)\to X$ is
continuous. Since $X\in AE(0)$, by Theorem~\ref{main}, there exists
a $P_c$-valued retraction $T\colon Y\to P_c(X)$. Then $r=b\circ
T\colon Y\to X$ is a retraction.
\end{proof}

\begin{lem}
Let $X\subset Y$ and $u\colon C(X)\to C(Y)$ be a regular extension
operator with compact supports. Suppose every closed bounded subset
of $X$ is compact. Then there exists a map $T_c\colon P_c(Y)\to
P_c(X)$ $($resp., $T_c^*\colon P_c^*(Y)\to P_c^*(X)$$)$ such that
$P_c(i)\circ T_c$ $($resp., $P_c^*(i)\circ T_c^*$$)$ is a
retraction, where $i\colon X\to Y$ is the embedding of $X$ into $Y$.
\end{lem}

\begin{proof}
For every $\mu\in P_c(Y)$ define $T_c(\mu)\colon C(X)\to\mathbb R$
by $T_c(\mu)(f)=\mu(u(f))$, $f\in C(X)$. Obviously, each $T_c(\mu)$
is linear. Let us show that $T_c(\mu)\in P_c(X)$ for all $\mu\in
P_c(Y)$. Since $u$ has compact supports, the map $T\colon Y\to
P_c(X)$ generated by $u$ is continuous. Hence, $T\big(s(\mu)\big)$
is a compact subset of $P_c(X)$ (recall that $s(\mu)\subset Y$ is
compact). Then by \cite{a1} (see also \cite[Proposition 3.1]{vv}),
$H(\mu)=\overline{\cup\{s(T(y)):y\in s(\mu)\}}$ is closed and
bounded in $X$, and hence compact. Let us show that the support of
$T_c(\mu)$ is compact. That will be done if we prove that
$s(T_c(\mu))\subset H(\mu)$. To this end, let $f(H(\mu))=0$ for some
$f\in C(X)$. Consequently, $T(y)(f)=0$ for all $y\in s(\mu)$. So,
$u(f)(s(\mu))=0$. The last equality means that $T_c(\mu)(f)=0$.
Hence, each $T_c(\mu)$ has a compact support and $T_c$ is a map from
$P_c(Y)$ to $P_c(X)$. It is easily seen that
$P_c(i)\big(T_c(\mu)\big)=\mu$ for all $\mu\in
P_c(i)\big(P_c(X)\big)$. Therefore, $P_c(i)\circ T_c$ is a
retraction from $P_c(Y)$ onto $P_c(i)\big(P_c(X)\big)$.

Now, we consider the linear operators $T_c^*(\nu)\colon
C^*(X)\to\mathbb R$, $T_c^*(\nu)(h)=\nu(u(h))$ with $\nu\in
P_c^*(Y)$ and $h\in C^*(X)$. Observed that $u(h)\in C^*(Y)$ for
$h\in C^*(X)$ because $u$ is a regular operator, so the above
definition is correct. To show that $T_c^*$ is a map from $P_c^*(Y)$
to $P_c^*(X)$, for every $\nu\in P_c^*(Y)$ take the unique $\mu\in
P_c(Y)$ with $j_Y(\mu)=\nu$. Then $s(\mu)=s^*(\nu)$ according to
Proposition 2.1. Hence, $T_c^*(\nu)(h)=0$ provided $h\in C^*(X)$
with $h|s\big(T_c(\mu)\big)=0$. So, the support of $T_c^*(\nu)$ is
contained in $s\big(T_c(\mu)\big)$. This means that $T_c^*$ maps
$P_c^*(Y)$ into $P_c^*(X)$. Moreover, one can show that
$P_c^*(i)\circ T_c^*$ is a retraction.
\end{proof}

Ditor and Haydon \cite{dh} proved that if $X$ is a compact space,
then $P(X)$ is an absolute retract if and only if $X$ is a Dugundji
space of weight $\leq\aleph_1$. A similar result concerning the
space of all $\sigma$-additive probability measures was established
by Banakh-Chigogidze-Fedorchuk \cite{bcf}. Next theorem shows that
the same is true when $P_c(X)$ or $P_c^*(X)$ is an $AR$.

\begin{thm}
For a space $X$ the following are equivalent:
\begin{itemize}
\item[(i)] $P_c(X)$ $($resp., $P_c^*(X)$$)$ is an absolute retract;
\item[(ii)] $P_c(X)$ $($resp., $P_c^*(X)$$)$ is an $AE(0)$;
\item[(iii)] $X$ is a Dugundji space of weight $\leq\aleph_1$.
\end{itemize}
\end{thm}

\begin{proof}
$(i)\Rightarrow (ii)$ This implication is trivial because every $AR$
is an $AE(0)$.

$(ii)\Rightarrow (iii)$ It suffices to show that $X$ is compact.
Indeed, then both $P_c(X)$ and $P_c^*(X)$ are $AE(0)$ and coincide
with $P(X)$. So, by Corollary~\ref{AR}, $P(X)$ is an $AR$. Applying
the mentioned above result of Ditor-Haydon, we obtain that $X$ is a
Dugundji space of weight $\leq\aleph_1$.

Suppose $X$ is not compact. Since $P_c(X)$ (resp., $P_c^*(X)$) is an
$AE(0)$-space, it is realcompact. Hence, so is $X$ as a closed
subset of $P_c(X)$ (resp., $P_c^*(X)$). Consequently, $X$ is not
pseudocompact (otherwise it would be compact), and there exists a
closed $C$-embedded subset $Y$ of $X$ homeomorphic to $\mathbb N$
(see the proof of Proposition 2.6). Since $Y$ is an $AE(0)$,
according to Theorem~\ref{main}, there exists a regular extension
operator $u\colon C(Y)\to C(X)$ with compact supports. Then, by
Lemma 4.4, $P_c(Y)$ (resp., $P_c^*(Y)$) is homeomorphic to a retract
of $P_c(X)$ (resp., $P_c^*(X)$). Hence, one of the spaces $P_c(Y)$
and $P_c^*(Y)$ is an $AE(0)$ (as a retract of an $AE(0)$-space).
Suppose $P_c^*(Y)\in AE(0)$. Since $P_c^*(Y)$ is second countable,
this implies $P_c^*(Y)$ is \v{C}ech-complete. Hence, by Proposition
2.6, $Y$ is pseudocompact, a contradiction. If $P_c(Y)\in AE(0)$,
then $P_c(Y)$ is metrizable according to a result of Chigogidze
\cite{ch2} stating that every $AE(0)$-space whose points are
$G_\delta$-sets is metrizable (the points of $P_c(Y)$ are $G_\delta$
because $j_Y\colon P_c(Y)\to P_c^*(Y)$ is an one-to-one surjection
and  $P_c^*(Y)$ is metrizable). But  by Proposition 2.5(ii),
$P_c(Y)$ is metrizable only if $Y$ is compact and metrizable. So, we
have again a contradiction.

$(iii)\Rightarrow (i)$ This implication follows from the stated
above result of Ditor and Haydon \cite{dh}.
\end{proof}

%%%%%%%%%%%%%%%%%%%%%%%%%%%%%%%%%%%%%%%%%%%%%%%%%%%%%%%%%%%%%%%%%%%%
\section{Properties preserved by Milyutin maps}

In this section we show that some topological properties are
preserved under Milyutin maps. Let $\mathfrak{F}$ be a family of
closed subsets of $X$. We say that {\em $X$ is collectionwise normal
with respect to $\mathfrak{F}$} if for every discrete family
$\{F_\alpha:\alpha\in A\}\subset\mathfrak{F}$  there exists a
discrete family $\{V_\alpha:\alpha\in A\}$ of open in $X$ sets with
$F\alpha\subset V_\alpha$ for each $\alpha\in A$. When $X$ is
collectionwise normal with respect to the family of all closed
subsets, it is called collectionwise normal.

\begin{thm}
Every weakly Milyutin map preserves paracompactness and
collectionwise normality.
\end{thm}

\begin{proof}
Let $f\colon X\to Y$ be a weakly Milyutin map and $u\colon C^*(X)\to
C^*(Y)$ a regular averaging operator for $f$ with compact supports.

Suppose $X$ is collectionwise normal, and let $\{F_\alpha:\alpha\in
A\}$ be a discrete family of closed sets in $Y$. Then
$\{f^{-1}(F_\alpha):\alpha\in A\}$ is a discrete collection of
closed sets in $X$. So, there exists a discrete family
$\{V_\alpha:\alpha\in A\}$ of open sets in $X$ with
$f^{-1}(F_\alpha)\subset V_\alpha$, $\alpha\in A$. Let
$V_0=X-\cup\{f^{-1}(F_\alpha):\alpha\in A\}$ and
$\gamma=\{V_\alpha:\alpha\in A\}\cup\{V_0\}$. Since $\gamma$ is a
locally finite open cover of $X$ and $X$ is normal (as
collectionwise normal), there exists a partition of unity
$\xi=\{h_\alpha:\alpha\in A\}\cup\{h_0\}$ on $X$ subordinated to
$\gamma$ such that $h_\alpha\big(f^{-1}(F_\alpha)\big)=1$ for every
$\alpha$. Observe that $h_{\alpha(1)}(x)+h_{\alpha(2)}(x)\leq 1$ for
any $\alpha(1)\neq \alpha(2)$ and any $x\in X$. So,
$u(h_{\alpha(1)})(y)+u(h_{\alpha(2)})(y)\leq 1$ for all $y\in Y$.
This yields that $\{u(h_\alpha)^{-1}\big((1/2,1]\big):\alpha\in A\}$
is a disjoint open family in $Y$. Moreover, $F_\alpha\subset
u(h_\alpha)^{-1}\big((1/2,1]\big)$ for every $\alpha$. Therefore,
$Y$ is collectionwise normal (see \cite[Theorem 5.1.17]{e}).

Let $X$ be paracompact and $\omega$ an open cover of $Y$. So, there
exists a locally finite open cover $\gamma $ of $X$ which an
index-refinement of $f^{-1}(\omega)$. Let $\xi$ be a partition of
unity on $X$ subordinated to $\gamma$. It is easily seen that
$u(\xi)$ is a partition of unity on $Y$ subordinated to $\omega$.
Hence, by \cite{m2}, $Y$ is paracompact.
\end{proof}

\begin{cor}
Let $f:X\to Y$ be a weakly Milyutin map and $X$  a $($completely$)$
metrizable space. Then $Y$ is also $($completely$)$ metrizable.
\end{cor}

\begin{proof}
Let $T\colon Y\to P_c^*(X)$ be a map associated with $f$. Then
$\phi=\Phi_X^*\circ T\colon Y\to X$ is a lsc compact-valued map (see
Lemma 3.3 for the map $\Phi_X^*$) such that $\phi(y)\subset
f^{-1}(y)$ for every $y\in Y$. Since $Y$ is paracompact (by Theorem
4.1), we can apply Michael's selection theorem \cite{m3} to find an
upper semi-continuous (br., usc) compact-valued selection
$\psi\colon Y\to X$ for $\phi$ (recall that $\psi$ is usc provided
the set $\{y\in Y:\psi(y)\cap F\neq\varnothing\}$ is closed in $Y$
for every closed $F\subset X$). Then $f|X_1:X_1\to Y$ is a perfect
surjection, where $X_1=\cup\{\psi(y):y\in Y\}$. Hence, $Y$ is
metrizable as a perfect image of a metrizable space.

If $X$ is completely metrizable, then so is $Y$. Indeed, by
\cite[Theorem 1.2]{at}, there exists a closed subset $X_0\subset X$
such that $f|X_0:X_0\to X$ is an open surjection. Then $Y$ is
complete (as a metric space being an open image of a complete metric
space).
\end{proof}

\begin{pro}
Let $f\colon X\to Y$ be a weakly Milyutin map with $X$ being a
product of metrizable spaces. Then we have:
\begin{itemize}
\item[(i)] The closure of any family of $G_\delta$-sets in $X$ is a
zero-set in $X$;
\item[(ii)] $X$ is collectionwise normal with respect to the family
of all closed $G_\delta$-sets in $X$.
\end{itemize}
\end{pro}

\begin{proof}Let $X=\prod\{X_\gamma:\gamma\in\Gamma\}$, where each
$X_\gamma$ is metrizable. Suppose $u\colon C^*(X)\to C^*(Y)$ is a
regular averaging operator for $f$ with compact supports.

(i) Let $G$ be a union of $G_\delta$-sets in $Y$. Then so is
$f^{-1}(G)$ in $X$ and, by \cite[Corollary]{k}, there exists $h\in
C^*(X)$ with $h^{-1}(0)=\overline{f^{-1}(G)}$. Since $h(T(y))=0$ for
each $y\in G$, $u(h)(G)=0$. On the other hand, $\inf\{h(x):x\in
T(y)\}>0$ for $y\not\in\overline{G}$. Hence, $u(h)(y)>0$ for any
$y\not\in\overline{G}$. Consequently, $u(h)^{-1}(0)=\overline{G}$.

(ii) Let $\{F_\alpha:\alpha\in A\}$ be a discrete family of closed
$G_\delta$-sets in $Y$. Then so is the family
$\{H_\alpha=f^{-1}(F_\alpha):\alpha\in A\}$ in $X$. Moreover, by
(i), each $F_\alpha$ is a zero-set in $Y$, hence $H_\alpha$ is a
zero-set in $X$.

We can assume that $\Gamma$ is uncountable (otherwise $X$ is
metrizable and the proof follows from Theorem 5.1). Consider the
$\Sigma$-product $\Sigma(a)$ of all $X_\gamma$ with a base-point
$a\in X$. Since $\Sigma(a)$ is $G_\delta$-dense in $X$ (i.e., every
$G_\delta$-subset of $X$ meets $\Sigma(a)$), $\Sigma(a)$ is
$C$-embedded in $X$ \cite{t} and

\medskip\noindent
(7)\hspace{0.3cm} $H_\alpha=\overline{H_\alpha\cap\Sigma(a)}$ for
any $\alpha$.

\smallskip\noindent
Because $\Sigma(a)$ is collectionwise normal \cite{g}, there exists
a discrete family $\{W_\alpha:\alpha\in A\}$ of open subsets of
$\Sigma(a)$ such that $H_\alpha\cap\Sigma(a)\subset W_\alpha$,
$\alpha\in A$. Let
$W_0=\Sigma(a)-\cup\{H_\alpha\cap\Sigma(a):\alpha\in A\}$.  Choose a
partition of unity $\{h_\alpha:\alpha\in A\}\cup\{h_0\}$ in
$\Sigma(a)$ subordinated to the locally finite cover
$\{W_\alpha:\alpha\in A\}\cup\{W_0\}$ of $\Sigma(a)$ such that
$h_\alpha\big(H_\alpha\cap\Sigma(a)\big)=1$ for each $\alpha$. Since
$\Sigma(a)$ is $C$-embedded in $X$, each $h_\alpha$  can be extended
to a function $g_\alpha$ on $X$. Because of (7),
$g_\alpha(H_\alpha)=1$, $\alpha\in A$. The density of $\Sigma(a)$ in
$X$ implies that $g_{\alpha(1)}(x)+g_{\alpha(2)}(x)\leq 1$ for any
$\alpha(1)\neq \alpha(2)$ and any $x\in X$. As in the proof of
Theorem 5.1, this implies that $F_\alpha\subset
U_\alpha=u(g_\alpha)^{-1}\big((1/2,1]\big)$ and the family
$\{U_\alpha:\alpha\in A\}$ is disjoint. Then, as in the proof of
\cite[Theorem 5.1.17]{e}, there exists a discrete family
$\{V_\alpha:\alpha\in A\}$ of open subsets of $Y$ with
$F_\alpha\subset V_\alpha$, $\alpha\in A$.
\end{proof}

A space $X$ is called {\em $k$-metrizable} \cite{sc1} if there
exists a {\em $k$-metric on $X$}, i.e., a non-negative real-valued
function $d$ on $X\times\mathcal{RC}(X)$, where $\mathcal{RC}(X)$
denotes the family of all regularly closed subset of $X$ (i.e.,
closed sets $F\subset X$ with $F=\overline{~int_X(F)}$) satisfying
the following conditions:
\begin{itemize}
\item[(K1)] $d(x,F)=0$ iff $x\in F$ for every $x\in X$ and
$F\in\mathcal{RC}(X)$;
\item[(K2)]  $F_1\subset F_2$ implies $d(x,F_2)\leq d(x,F_1)$ for
every $x\in X$;
\item[(K3)] $d(x,F)$ is continuous with respect to $x$ for every
$F\in\mathcal{RC}(X)$;
\item[(K4)] $d\big(x,\overline{\cup\{F_\alpha:\alpha\in
A\}}\big)=\inf\{d(x,F_\alpha):\alpha\in A\}$ for every $x\in X$ and
every increasing linearly ordered by inclusion family
$\{F_\alpha\}_{\alpha\in A}\subset\mathcal{RC}(X)$.
\end{itemize}

If $\mathcal K(X)$ is a family of closed subsets of $X$, then a
function $d\colon X\times\mathcal K(X)\to\mathcal R$ satisfying
conditions $(K1)- (K3)$ with $\mathcal{RC}(X)$ replaced by $\mathcal
K(X)$ is called {\em a monotone continuous annihilator} of the
family $\mathcal K(X)$ \cite{dr}. When $\mathcal K(X)$ consists of
all zero sets in $X$, then any monotone continuous annihilator is
said to be a {\em $\delta$-metric on $X$} \cite{dr}. The well known
notion of stratifiability \cite{bo} can be express as follows: $X$
is stratifiable iff there exists a monotone continuous annihilator
on $X$ for the family of all closed subsets of $X$.

A space $X$ is perfectly $k$-normal \cite{sc2} provided every
$F\in\mathcal{RC}(X)$ is a zero-set in $X$.

\begin{thm}
Every weakly Milyutin map $f:X\to Y$ preserves the following
properties: stratifiability, $\delta$-metrizability, and perfectly
$k$-normality. If, in addition,
$cl_X\big(f^{-1}(U)\big)=f^{-1}\big(cl_Y(U)\big)$ for every open
$U\subset Y$, then $f$ preserves $k$-metrizability.
\end{thm}

\begin{proof}
We consider only the case $f$ satisfies the additional condition
which is denoted by (s) (the proof of the other cases is similar).
Let $u:C^*(X)\to C^*(Y)$ be a regular averaging operator for $f$
having compact supports, and $d(x,F)$ be a $k$-metric on $X$. We may
assume that $d(x,F)\leq 1$ for any $x\in X$ and
$F\in\mathcal{RC}(X)$, see \cite{sc1}. Let
$F_G=cl_X\big(f^{-1}(int_Y(G))\big)$ for each $G\in\mathcal{RC}(Y)$,
and define $h_G(x)=d(x,F_G)$. Consider the function
$\rho:Y\times\mathcal{RC}(Y)\to\mathbb R$, $\rho(y,G)=u(h_G)(y)$. We
are going to check that $\rho$ is a $k$-metric on $Y$.

Suppose $G(1), G(2)\in\mathcal{RC}(Y)$ and $G(1)\subset G(2)$. Then
$F_{G(1)}\subset F_{G(2)}$, so $h_{G(2)}\leq h_{G(1)}$.
Consequently, $\rho(y,G(2))\leq\rho(y,G(1))$ for any $y\in Y$. On
the other hand, obviously, $\rho(y,G)$ is  continuous with respect
to $y$ for every $G\in\mathcal{RC}(Y)$. Hence, $\rho$ satisfies
conditions $(K2)$ and $(K3)$.

Suppose $G\in\mathcal{RC}(Y)$. Then  $s^*\big(T(y)\big)\subset
f^{-1}(y)\subset F_G$ for every $y\in int_Y(G)$, where $T\colon Y\to
P_c^*(X)$ is the associated map to $f$ generated by $u$.
Consequently, $h_G|s^*\big(T(y)\big)=0$ which implies $u(h_G)(y)=0$,
$y\in int_Y(G)$. On the other hand, if $y\not\in G$, then
$s^*\big(T(y)\big)\cap F_G=\varnothing$ and $h_G(x)>0$ for all $x\in
s^*\big(T(y)\big)$. Since $u(h_G)(y)\geq\inf\{h_G(x):x\in
s^*\big(T(y)\big)\}$ (recall that $u$ is an averaging operator for
$f$), $u(h_G)(y)>0$. Hence, $u(h_G)(y)=\rho(y,G)=0$ iff $y\in G$, so
$\rho$ satisfies condition $(K1)$.

 To check condition $(K4)$, suppose $\{G(\alpha):\alpha\in A\}\subset\mathcal{RC}(Y)$ is
 an increasing linearly ordered by inclusion family and
 $G=cl_Y\big(\cup\{G(\alpha):\alpha\in A\}\big)$. Using that
 $f$ satisfies condition (s), we have $F_G=cl_X\big(\cup\{F_{G(\alpha)}:\alpha\in
 A\}\big)$.  Since $\{F_{G(\alpha)}:\alpha\in A\}$ is also increasing and linearly ordered by inclusion,
 according to condition $(K4)$,
 $h_G(x)=\inf\{h_{G(\alpha)}(x):\alpha\in A\}$ for every $x\in X$.
 Let $y\in Y$ and $\epsilon>0$. Then for every $x\in X$ there exists
 $\alpha_x\in A$ such that $h_{G(\alpha_x)}(x)<h_G(x)+\epsilon$.
 Choose a neighborhood $V(x)$ of $x$ in $X$ such that
 $h_{G(\alpha_x)}(z)<h_G(z)+\epsilon$ for all $z\in V(x)$. Since
 $s^*\big(T(y)\big)$ is compact, it can be covered by finitely many
 $V(x(i))$, $i=1,..,n$, with $x(i)\in s^*\big(T(y)\big)$. Let
 $\beta=\max\{\alpha_{x(i)}:i\leq n\}$. Then
 $h_{G(\beta)}(x)<h_G(x)+ \epsilon$ for all $x\in
 s^*\big(T(y)\big)$. The last equality yields
 $\rho(y,G(\beta))\leq\rho(y,G)+\epsilon$ because
 $u(h_{G(\beta)})(y)$ and $u(h_G)(y)$ depend only on the
 restrictions $h_{G(\beta)}|s^*\big(T(y)\big)$ and $h_G|s^*\big(T(y)\big)$, respectively.
 Thus, $\inf\{\rho(y,G(\alpha)):\alpha\in A\}\leq\rho(y,G)$.
 The inequality $\rho(y,G)\leq\inf\{\rho(y,G(\alpha)):\alpha\in A\}$
 is obvious because $G$ contains each $G(\alpha)$, so
 $\rho$ satisfies condition $(K4)$. Therefore, $Y$ is
 $k$-metrizable.
\end{proof}

 Next corollary provides a positive answer to a question of
 Shchepin \cite{sc3}.

\begin{cor}
Every $AE(0)$-space is $k$-metrizable.
\end{cor}

\begin{proof}
Let $X$ be an $AE(0)$-space of weight $\tau$. By \cite[Theorem
4]{ch2}, there exists a surjective $0$-soft map $f\colon\mathbb
N^\tau\to X$. Since $\mathbb N^\tau\in AE(0)$ (as a product of
$AE(0)$-space) and every $0$-soft map between $AE(0)$-spaces is
functionally open \cite[Theorem 1.15]{ch2}, $f$ satisfies condition
(s) from the previous theorem. On the other hand, $\mathbb N^\tau$
is $k$-metrizable as a product of metrizable spaces \cite[Theorem
15]{sc1}. Hence, the proof follows from Proposition 3.12 and Theorem
5.4.
\end{proof}

%%%%%%%%%%%%%%%%%%%%%%%%%%%%%%%%%%%%%%%%%%%%%%%%%%%%%%%%%%%%%%%%%%%%%%%%%%%%%%%%%%%%%%%%%%%%%%%%%%%%%

\bigskip

\end{document}